\documentclass[secnum,leqno]{rims-bessatsu}
\usepackage{amssymb, amsmath, amscd}
\usepackage{graphics}
\usepackage[dvips]{graphicx}
\usepackage{eepic}
\usepackage{epic}
\newtheorem{thm}{Theorem}[section]

\newtheorem{lmm}[thm]{Lemma}

\newtheorem{prp}[thm]{Proposition}

\theoremstyle{definition}

\theoremstyle{remark}
\newtheorem*{ack}{Acknowledgement}
\newcommand{\FonP}{$(\mathrm{P}_{m})$}
\newcommand{\Order}{\mathcal{O}}
\newcommand{\maxid}{\mathfrak{p}}
\newcommand{\into}{\hookrightarrow}
\newcommand{\onto}{\twoheadrightarrow}
\newcommand{\isomto}{\overset{\sim}{\to}}
\newcommand{\compose}{\circ}
\newcommand{\tensor}{\otimes}
\newcommand{\ctensor}{\mathbin{\hat{\tensor}}}
\newcommand{\algcl}[1]{\overline{#1}}

\newcommand{\Z}{\mathbb{Z}}
\newcommand{\Q}{\mathbb{Q}}
\newcommand{\F}{\mathbb{F}}
\newcommand{\ab}{\mathrm{ab}}
\newcommand{\sep}{\mathrm{sep}}
\newcommand{\ur}{\mathrm{ur}}
\newcommand{\id}{\mathrm{id}}
\newcommand{\alg}[1]{\mathbf{#1}}
\newcommand{\Sch}[1]{\mathrm{Sch} / #1}
\newcommand{\Perf}[1]{\mathrm{Perf} / #1}
\newcommand{\fpqc}[1]{(\mathrm{Sch} / #1)_{\mathrm{fpqc}}}
\newcommand{\pfpqc}[1]{(\mathrm{Perf} / #1)_{\mathrm{fpqc}}}
\newcommand{\sh}[1]{\mathrm{Ab}(\mathrm{Sch} / #1)_{\mathrm{fpqc}}}
\newcommand{\psh}[1]{\mathrm{Ab}(\mathrm{Perf} / #1)_{\mathrm{fpqc}}}
\newcommand{\Ga}{\mathbf{G}_{a}}
\newcommand{\Gm}{\mathbf{G}_{m}}
\DeclareMathOperator{\Gal}{Gal}
\DeclareMathOperator{\Hom}{Hom}

\DeclareMathOperator{\Ker}{Ker}
\let\Im\relax
\DeclareMathOperator{\Im}{Im}
\DeclareMathOperator{\Coker}{Coker}
\DeclareMathOperator{\Ext}{Ext}
\DeclareMathOperator{\Res}{Res}
\DeclareMathOperator{\Spec}{Spec}
\DeclareMathOperator{\Grn}{\mathbf{Grn}}
\DeclareMathOperator{\Cores}{Cores}
\title{A refinement of the local class field theory
of Serre and Hazewinkel}
\author{\textsc{Takashi Suzuki}\footnote{Department of Mathematics, 
University of Chicago,
5734 S University Ave,
Chicago, IL 60637
\newline e-mail: \texttt{suzuki@math.uchicago.edu}} and 
\textsc{Manabu Yoshida}\footnote{Graduate School of Mathematics, 
Kyushu University, Fukuoka
819-0395, Japan.\newline e-mail: \texttt{m-yoshida@kyushu-u.ac.jp}}
}
\AuthorHead{Takashi Suzuki and Manabu Yoshida}         
\support{The second author is 
supported by the Japan Society for the Promotion of Science Research 
Fellowships for Young Scientists}  
\VolumeNo{32}           
\YearNo{2012}           
\PagesNo{163--191}      
\renewcommand{\volyr}{2012}
\communication{Received March 31, 2012.
Revised May 25, 2012.}      
\begin{document}
%

\maketitle

\begin{abstract}      
	We give a refinement of the local class field theory of Serre and Hazewinkel.
	This refinement allows the theory to treat extensions that are not necessarily totally ramified.
	Such a refinement was obtained and used in the authors' paper on Fontaine's property (P$_{m}$),
	where the explanation had to be rather brief.
	In this paper, we give a complete account,
	from necessary knowledge of an appropriate Grothendieck site
	to the details of the proof.
	We start by reviewing the local class field theory of Serre and Hazewinkel.
\end{abstract}

\tableofcontents      


\section{Introduction}

Let $K$ be a complete discrete valuation field
with perfect residue field $k$ of characteristic $p > 0$ and
let $K^{\mathrm{ab}}$ be the maximal abelian extension of $K$.
When $k$ is finite, the usual local class field theory gives a canonical homomorphism
	\[
		K^{\times} \to \Gal(K^{\ab} / K),
	\]
which induces a commutative diagram
	\[
		\begin{CD}
				0
			@>>>
				U_{K}
			@>>>
				K^{\times}
			@>>>
				\Z
			@>>>
				0
			\\
			@. @V \wr VV @VVV @VVV @.
			\\
				0
			@>>>
				T(K^{\ab} / K)
			@>>>
				\Gal(K^{\ab} / K)
			@>>>
				\Gal(k^{\ab} / k)
			@>>>
				0,
		\end{CD}
	\]
where $U_{K}$ is the group of units of $K$ and
$T$ denotes the inertia group.
Serre (\cite{Ser61}) gave an analogue of this theory
for the case where the residue field $k$ is algebraically closed.
For this, he developed the theory of proalgebraic groups
(more precisely, pro-quasi-algebraic groups)
and their fundamental groups in his paper \cite{Ser60}.
There the group of units $U_{K}$ was viewed as a proalgebraic group over the residue field $k$.
We denote this proalgebraic group by $\alg{U}_{K}$
and its fundamental group by $\pi_{1}^{k}(\alg{U}_{K})$.
He proved the existence of a canonical isomorphism
	\[
			\pi_{1}^{k}(\alg{U}_{K})
		\isomto
			\Gal(K^{\ab} / K).
	\]
This is the local class field theory of Serre.
Later Hazewinkel generalized this theory to the case
where the residue field $k$ is a perfect field.
He defined the proalgebraic group of units $\alg{U}_{K}$ over $k$
and its fundamental group $\pi_{1}^{k}(\alg{U}_{K})$ in a similar way
in \cite[Appendice]{DG70},
and proved the existence of a canonical isomorphism
	\[
			\pi_{1}^{k}(\alg{U}_{K})
		\isomto
			T(K^{\ab} / K).
	\]
This is the local class field theory of Hazewinkel.

In this paper, we extend the local class field theory of Serre and Hazewinkel
so as to describe the whole group $\Gal(K^{\ab} / K)$
in the case where the residue field $k$ is a general perfect field.
For this, we view the multiplicative group $K^{\times}$ of $K$ as a group scheme
(more precisely, a perfect group scheme, on which the Frobenius is an isomorphism),
denoted by $\alg{K}^{\times}$,
which is isomorphic to the direct product of $\alg{U}_{K}$ and
the discrete group scheme $\mathbb{Z}$ over $k$.
We will define its fundamental group $\pi_{1}^{k}(\alg{K}^{\times})$ in Section \ref{Sec:Fpqc}
using the Ext functor for the category of sheaves on a version of the fpqc site of $k$.
Our main result is the following.

\begin{thm} \label{Th:Main}
	For a complete discrete valuation field $K$ with perfect residue field $k$,
	there exists a canonical isomorphism
		\[
				\pi_{1}^{k}(\alg{K}^{\times})
			\isomto
				\Gal(K^{\ab} / K)
		\]
	with a commutative diagram
		\begin{equation} \label{Eq:LCFT:Compati}
			\begin{CD}
					0
				@>>>
					\pi_{1}^{k}(\alg{U}_{K})
				@>>>
					\pi_{1}^{k}(\alg{K}^{\times})
				@>>>
					\pi_{1}^{k}(\Z)
				@>>>
					0
				\\
				@. @VV \wr V @VV \wr V @VV \wr V @.
				\\
					0
				@>>>
					T(K^{\ab} / K)
				@>>>
					\Gal(K^{\ab} / K)
				@>>>
					\Gal(k^{\ab} / k)
				@>>>
					0,
			\end{CD}
		\end{equation}
	where the left vertical isomorphism is the one given by the local class field theory of Hazewinkel
	and the right vertical isomorphism is the natural one
	(see the end of Section \ref{Sec:Thick}) times $-1$.
\end{thm}

In the case the residue field $k$ is the finite field $\F_{q}$ with $q$ elements,
we have a natural homomorphism
$K^{\times} \to \pi_{1}^{k}(\alg{K}^{\times})$ such that
the composite map
	$
			K^{\times}
		\to
			\pi_{1}^{k}(\alg{K}^{\times})
		\isomto
			\Gal(K^{\ab} / K)
	$
coincides with the canonical map of the usual local class field theory times $-1$,
which sends a prime element to an automorphism that acts on $k^{\ab} = \algcl{\F_{q}}$
by the $q^{-1}$-th power map
(see the paragraph after Proposition \ref{Prop:BC}).

Actually Theorem \ref{Th:Main} was previously formulated and proved in the authors' paper
on Fontaine's property \FonP\ (\cite[Prop.\ 4.1]{SY10}).
The explanation in that paper, however, had to be rather brief,
since the details of this theorem are too complicated,
so that detailed explanation could destroy the organization of that paper.
Giving precise formulation of Theorem \ref{Th:Main} and proving it
is the subject of this paper.

The organization of this paper is as follows.
In Section \ref{Sec:SH}, we give a review of the local class field theory of Serre and Hazewinkel.
In formulating and proving the above refinement of this theory,
several difficulties naturally appear.
The beginning of Section \ref{Sec:Fpqc} is devoted to
an explanation about these problems and to an outline of the way we take to solve them.
In Section \ref{Sec:Fpqc}, we define a version of the fpqc site of a perfect field $k$
and develop a general theory on it as preparation for the next section.
Section \ref{Sec:FormPf} is the local class field theory
for a complete discrete valuation field $K$ with perfect residue field $k$.
We construct some sheaves associated with $K$ and its finite extensions,
and prove Theorem \ref{Th:Main} as well as some auxiliary results that are needed in \cite[\S 4]{SY10}.
Detailed explanation about the organization of Sections \ref{Sec:Fpqc} and \ref{Sec:FormPf}
is given at the beginning of Section \ref{Sec:Fpqc}.

\begin{ack}
	The authors would like to thank Professor Kato and Professor Taguchi
	for having helpful discussions,
	and to Professor Fesenko for suggesting relation between his work \cite{Fes93} and our work.
	They also thank to the referee and Professor Suwa
	for reading a draft of the paper and giving comments for it.
\end{ack}


\section{Review of the local class field theory of Serre and Hazewinkel}
\label{Sec:SH}

In this section, we review the local class field theory of Serre and Hazewinkel.
We first discuss the part that is due to Serre.
Let $k$ be an algebraically closed field of characteristic $p > 0$.

First we recall quasi-algebraic groups and proalgebraic groups over $k$,
as well as their fundamental groups,
from Serre's paper \cite{Ser60}.
Roughly speaking, a quasi-algebraic group is the ``perfection'' of an algebraic group.
Let us make this precise.
For a commutative algebraic group $A$ over $k$ and a non-negative integer $i$,
let $A^{(i)}$ be the algebraic group $A$ over $k$
whose structure morphism is replaced by
the composite of the original structure morphism $A \to \Spec k$ and
the $p^{i}$-th power Frobenius morphism $\Spec k \isomto \Spec k$.
We denote by $A^{(\infty)}$ the projective limit of the sequence of algebraic groups
	\[
		\cdots \to A^{(2)} \to A^{(1)} \to A,
	\]
where the morphism $A^{(i + 1)} \to A^{(i)}$ is the Frobenius morphism.
We know that $A$ and $A^{(\infty)}$ are group schemes over $k$
having the same underlying topological space.
We also know that $A^{(\infty)}$ is perfect, namely the Frobenius gives an isomorphism on $A^{(\infty)}$.
A commutative quasi-algebraic group over $k$ (\cite[\S 1]{Ser60})
is a group scheme of the form $A^{(\infty)}$
for some commutative algebraic group $A$ over $k$.
The category of commutative quasi-algebraic groups is an artinian abelian category
(\cite[\S 1, Prop.\ 5-6]{Ser60}).
Its procategory is the category of commutative proalgebraic groups defined by Serre (\cite[\S 2]{Ser60}).
This is an abelian category with enough projectives
(\cite[\S 2, Prop.\ 7 and \S 3, Prop.\ 1]{Ser60}).
The exactness of a sequence $A \to B \to C$ in the category of commutative proalgebraic groups is
equivalent to the exactness of the sequence $A(k) \to B(k) \to C(k)$ induced on the groups of $k$-rational points
(\cite[\S 1, Prop.\ 4-5]{Ser60}).
Define a functor $\pi_{0}^{k}$ from the category of commutative proalgebraic groups
to the category of profinite abelian groups
by taking the group of connected components
($=$ the maximal profinite quotient) (\cite[\S 5.1]{Ser60}).
This functor is right exact (\cite[\S 5, Prop.\ 2]{Ser60}).
For $i \ge 0$, the $i$-th left derived functor of $\pi_{0}^{k}$
is called the $i$-th homotopy group functor (\cite[\S 5, Def.\ 1]{Ser60}),
which we denote by $\pi_{i}^{k}$.
The functor $\pi_{1}^{k}$ is called the fundamental group functor.
Let $\Ext_{k}^{i}$ be the $i$-th Ext functor for the category of commutative proalgebraic groups.
Since
	$
			\injlim_{n \ge 1} \Hom_{k}(A, n^{-1} \Z / \Z)
		\cong
			\Hom(\pi_{0}^{k}(A), \Q / \Z)
	$
for any proalgebraic group $A$,
we have
	$
			\injlim_{n \ge 1} \Ext_{k}^{i}(A, n^{-1} \Z / \Z)
		\cong
			\Hom(\pi_{i}^{k}(A), \Q / \Z)
	$
for any $i \ge 0$.
When $i = 1$, this means that $\pi_{1}^{k}(A)$ classifies
surjective isogenies to $A$ with (pro-)finite constant kernels.

Let $K$ be a complete discrete valuation field with residue field $k$.
Then, as explained in Section 1 of Serre's paper \cite{Ser61},
the group of units $U_{K}$ of $K$ can be viewed as a proalgebraic group over $k$.
We denote this proalgebraic group by $\alg{U}_{K}$.
This group is affine (or equivalently, (pro-)linear) and connected (\cite[\S 1.3]{Ser61}).
The group of $k$-rational points of $\alg{U}_{K}$ is given by the abstract group $U_{K}$:
$\alg{U}_{K}(k) = U_{K}$.
Also for each $n \ge 0$,
the group of $n$-th principal units $U_{K}^{n}$ can be viewed as a proalgebraic group,
denoted by $\alg{U}_{K}^{n}$.
This is a proalgebraic subgroup of $\alg{U}_{K}$,
the quotient $\alg{U}_{K} / \alg{U}_{K}^{n}$ being an $n$-dimensional quasi-algebraic group.
We have $\alg{U}_{K} \isomto \projlim_{n \ge 0} \alg{U}_{K} / \alg{U}_{K}^{n}$,
which gives the proalgebraic structure for $\alg{U}_{K}$.
Also we have $\alg{U}_{K} / \alg{U}_{K}^{1} \cong \Gm^{(\infty)}$,
where $\Gm^{(\infty)}$ is the quasi-algebraic group
associated with the algebraic group $\Gm$ of invertible elements.
The Teichm\"{u}ller section $\Gm^{(\infty)} \into \alg{U}_{K}$ defines a splitting
$\alg{U}_{K} \cong \Gm^{(\infty)} \times \alg{U}_{K}^{1}$.
The main theorem of the local class field theory of Serre is the following.

\begin{thm}[\cite{Ser61}] \label{Th:SerreLCFT}
	Assume that $k$ is algebraically closed as above.
	Then there exists a canonical isomorphism
		\[
			\pi_{1}^{k}(\alg{U}_{K}) \isomto \Gal(K^{\ab} / K).
		\]
\end{thm}

\begin{proof}[Sketch of Proof]
	The construction of the isomorphism is as follows (\cite[\S 2]{Ser61}).
	For a finite Galois extension $L / K$ with Galois group $G$,
	let $\alg{U}_{L}$ be the proalgebraic group of units of $L$ over $k$.%
		\footnote{Note that the residue field of $L$ is $k$ since $k$ is algebraically closed,
		so that we can define $\alg{U}_{L}$ over $k$ in the same way we defined $\alg{U}_{K}$.}
	The norm map for $L / K$ induces a homomorphism of proalgebraic groups
	$N_{L / K} \colon \alg{U}_{L} \to \alg{U}_{K}$
	that is surjective (\cite[\S 2, Cor.\ to Prop.\ 1]{Ser61}).
	The Galois group $G$ acts on $\alg{U}_{L}$.
	Let $I_{G}$ be the augmentation ideal of the group ring $\Z[G]$
	and let $I_{G} \alg{U}_{L}$ be the product of the $G$-module $\alg{U}_{L}$ and the ideal $I_{G}$.

	We show that the sequence
		\begin{equation} \label{Eq:SerEx}
				0
			\to
				G^{\ab}
			\to
				\alg{U}_{L} / I_{G} \alg{U}_{L}
			\overset{N_{L / K}}{\to}
				\alg{U}_{K}
			\to
				0
		\end{equation}
	is an exact sequence of proalgebraic groups,
	where $G^{\ab}$ is the maximal abelian quotient of $G$
	viewed as a constant group over $k$ and
	the first map sends $\sigma \mapsto \sigma(\pi_{L}) / \pi_{L}$
	for a prime element $\pi_{L}$ of $L$.
	The group of $k$-rational points of the proalgebraic group $\Ker(N_{L / K}) / I_{G} \alg{U}_{L}$
	is the Tate cohomology group $\hat{H}^{-1}(G, U_{L})$
	(\cite[VIII, \S 1]{Ser79}).
	Consider the short exact sequence of $G$-modules
		\begin{equation} \label{Eq:ValExtGp}
				0
			\to
				U_{L}
			\to
				L^{\times}
			\to
				\Z
			\to
				0.
		\end{equation}
	The norm map $N_{L / K} \colon L^{\times} \to K^{\times}$ is surjective
	also as a morphism of abstract groups,
	so $\hat{H}^{0}(G, L^{\times}) = 0$.
	By Hilbert's theorem 90, we have $\hat{H}^{1}(G, L^{\times}) = 0$.
	Therefore we have
		\begin{equation} \label{Eq:Vanish}
				\hat{H}^{i}(G, L^{\times}) = 0
			\quad \text{for any} \quad
				i \in \Z
		\end{equation}
	by \cite[IX, Th.\ 8]{Ser79}.
	Hence we have
		$
				\hat{H}^{-1}(G, U_{L})
			\overset{\sim}{\leftarrow}
				\hat{H}^{-2}(G, \Z)
			\cong
				G^{\ab}
		$.
	This proves the exactness of the sequence \eqref{Eq:SerEx}.

	The homotopy long exact sequence induced by the short exact sequence \eqref{Eq:SerEx}
	gives an exact sequence
		\[
				\pi_{1}^{k}(\alg{U}_{K})
			\to
				\pi_{0}^{k}(G^{\ab})
			\to
				\pi_{0}^{k}(\alg{U}_{L} / I_{G} \alg{U}_{L}).
		\]
	We have $\pi_{0}^{k}(G^{\ab}) = G^{\ab}$.
	Since $\alg{U}_{L}$ is connected, so is $\alg{U}_{L} / I_{G} \alg{U}_{L}$,
	hence $\pi_{0}^{k}(\alg{U}_{L} / I_{G} \alg{U}_{L}) = 0$.
	Therefore we have a surjection $\pi_{1}^{k}(\alg{U}_{K}) \onto G^{\ab} = \Gal(L / K)^{\ab}$.
	Taking the limit in $L$, we have a surjection
	$\pi_{1}^{k}(\alg{U}_{K}) \onto \Gal(K^{\ab} / K)$.
	This is injective (``the existence theorem''; \cite[\S 4, Th.\ 1]{Ser61}).
	The isomorphism $\pi_{1}^{k}(\alg{U}_{K}) \isomto \Gal(K^{\ab} / K)$
	thus obtained is the one stated at the theorem.
\end{proof}

Now we review Hazewinkel's generalization of Serre's theory (\cite[Appendice]{DG70}).
Let $k$ be a perfect field.
Although he used the category of affine group schemes over $k$,
we instead use the categories of quasi-algebraic groups and proalgebraic groups over $k$
to make the discussion parallel to that of Serre.
Quasi-algebraic groups and proalgebraic groups over $k$ are defined
in a similar way as before.
The exactness of a sequence $A \to B \to C$ of commutative proalgebraic groups over $k$
is equivalent to the exactness of the sequence
$A(\algcl{k}) \to B(\algcl{k}) \to C(\algcl{k})$ induced on the groups of $\algcl{k}$-rational points.
The functor from the category of commutative proalgebraic groups over $k$
to the category of profinite abelian groups
taking the maximal proconstant quotient is denoted by $\pi_{0}^{k}$.
The left derived functors $\pi_{i}^{k}$ of $\pi_{0}^{k}$ are called the homotopy group functors
and $\pi_{1}^{k}$ is called the fundamental group functor.
The Pontryagin dual of $\pi_{i}^{k}(A)$ is given by
$\injlim_{n \ge 1} \Ext_{k}^{i}(A, n^{-1} \Z / \Z)$.
For a complete discrete valuation field $K$ with residue field $k$,
the proalgebraic group of units $\alg{U}_{K}$ over $k$ is defined similarly.
We have $\alg{U}_{K}(k) = U_{K}$.
Also we have $\alg{U}_{K}(\algcl{k}) = \hat{U}_{K}^{\ur} =$
the group of units of the completion of the maximal unramified extension $\hat{K}^{\ur}$ of $K$.
The main theorem of the local class field theory of Hazewinkel is the following.

\begin{thm}[{\cite[Appendice]{DG70}}] \label{Th:HazeLCFT}
	In the case $k$ is a general perfect field,
	there exists a canonical isomorphism
		\[
				\pi_{1}^{k}(\alg{U}_{K})
			\isomto
				T(K^{\ab} / K),
		\]
	where $T$ denotes the inertia group.%
		\footnote{Actually Hazewinkel used in \cite[V, \S 3, 4.2 and Appendice]{DG70}
		a slightly different functor $\gamma$
		to establish an isomorphism $\gamma(\alg{U}_{K}) \cong T(K^{\ab} / K)$.
		However, for a connected affine proalgebraic group $A$ (for example, $A = \alg{U}_{K}$),
		we have $\Hom(\gamma(A), N) \cong \Ext_{k}^{1}(A, N)$
		for any finite constant $N$ by \cite[V, \S 3, 4.2 Prop.]{DG70}
		and hence $\gamma(A) \cong \pi_{1}^{k}(A)$.
		Essentially the definition of $\gamma(\alg{U}_{K})$ is
		the $\Gal(\algcl{k} / k)$-coinvariants of $\pi_{1}^{\algcl{k}}(\alg{U}_{\hat{K}^{\ur}})$.
		Hence one step of the second proof of the theorem below
		can be interpreted as reproving $\gamma(\alg{U}_{K}) \cong \pi_{1}^{k}(\alg{U}_{K})$.}
\end{thm}

\begin{proof}[Sketch of Proof]
	For a finite totally ramified Galois extension $L / K$,
	the exact sequence \eqref{Eq:SerEx} is defined over $k$
	with $G^{\ab} = \Gal(L / K)^{\ab}$ viewed as a constant group
	(\cite[Appendice, \S 4.2]{DG70}).
	Thus we have a surjection $\pi_{1}^{k}(\alg{U}_{K}) \onto \Gal(L / K)^{\ab}$.
	For any infinite totally ramified Galois extension $L' / K$,
	we have a surjection $\pi_{1}^{k}(\alg{U}_{K}) \onto \Gal(L' / K)^{\ab}$
	by taking the limit over subfields $L \subset L'$ finite Galois over $K$.
	We can choose $L'$ so that $L' K^{\ur}$ is the separable closure of $K$
	(\cite[Appendice, \S 2.1]{DG70}).
	The composite map
		$
				\pi_{1}^{k}(\alg{U}_{K})
			\onto
				\Gal(L' / K)^{\ab}
			\onto
				T(K^{\ab} / K)
		$
	is independent of the choice of such $L'$ (\cite[Appendice, \S 6.2]{DG70}) and
	is an isomorphism (\cite[Appendice, \S 7.3]{DG70}).
	
	We give another proof of the theorem by reducing it to Serre's theorem \ref{Th:SerreLCFT}.
	If we apply this theorem for the completion $\hat{K}^{\ur}$ of the maximal unramified extension of $K$,
	we get an isomorphism
		\begin{equation} \label{Eq:IsomOverClosed}
				\pi_{1}^{\algcl{k}}(\alg{U}_{\hat{K}^{\ur}})
			\isomto
				\Gal((\hat{K}^{\ur})^{\ab} / \hat{K}^{\ur}).
		\end{equation}
	The proalgebraic group $\alg{U}_{\hat{K}^{\ur}}$
	can be obtained as the base extension of $\alg{U}_{K}$ from $k$ to $\algcl{k}$,
	so the absolute Galois group $\Gal(\algcl{k} / k)$ of $k$ acts on
	$\pi_{1}^{\algcl{k}}(\alg{U}_{\hat{K}^{\ur}})$.
	The group $\Gal(\algcl{k} / k)$ acts on $\Gal((\hat{K}^{\ur})^{\ab} / \hat{K}^{\ur})$ as well
	by lifting elements $\Gal(\algcl{k} / k)$ to $(\hat{K}^{\ur})^{\ab}$
	and then taking the conjugation action of them on $\Gal((\hat{K}^{\ur})^{\ab} / \hat{K}^{\ur})$.
	With these actions,
	the isomorphism \eqref{Eq:IsomOverClosed} is $\Gal(\algcl{k} / k)$-equivariant.
	We show that the $\Gal(\algcl{k} / k)$-coinvariants of
	the left-hand side (resp.\ the right-hand side) is
	$\pi_{1}^{k}(\alg{U}_{K})$ (resp.\ $T(K^{\ab} / K)$).
	The assertion for the right-hand side follows from the fact that
	the natural surjection $\Gal(K^{\sep} / K) \onto \Gal(\algcl{k} / k)$
	admits a section (\cite[\S 4.3, Exercises]{Ser02}).
	For the left-hand side,
	we use the natural spectral sequence
		$
				H^{i}(
					k,
					\Ext_{\algcl{k}}^{j}(
						\alg{U}_{\hat{K}^{\ur}}, \Q / \Z
					)
				)
			\Rightarrow
				\Ext_{k}^{i + j}(\alg{U}_{K}, \Q / \Z)
		$.
	We have
		$
				\Hom_{\algcl{k}}(\alg{U}_{\hat{K}^{\ur}}, \Q / \Z)
			=
				0
		$
	by the connectedness of $\alg{U}_{\hat{K}^{\ur}}$.
	Thus
		$
				H^{0}(
					k,
					\Ext_{\algcl{k}}^{1}(
						\alg{U}_{\hat{K}^{\ur}}, \Q / \Z
					)
				)
			=
				\Ext_{k}^{1}(\alg{U}_{K}, \Q / \Z)
		$.
	Hence the $\Gal(\algcl{k} / k)$-coinvariants of $\pi_{1}^{\algcl{k}}(\alg{U}_{\hat{K}^{\ur}})$
	is $\pi_{1}^{k}(\alg{U}_{K})$.
	Thus we have the required isomorphism
	by taking the $\Gal(\algcl{k} / k)$-coinvariants of the isomorphism \eqref{Eq:IsomOverClosed}.
\end{proof}


\section{The perfect fpqc site}
\label{Sec:Fpqc}

Now we want to formulate and prove Theorem \ref{Th:Main}.
Let us point out what we need for this.
As explained in Introduction,
we need to work in a category containing
both the proalgebraic group $\alg{U}_{K}$ over the perfect field $k$
and the discrete group scheme $\Z$,
to which the homotopy group functors $\pi_{i}^{k}$ should be extended.
We want to have an abelian category as such a category.
This causes a problem since the quotient of a proalgebraic group
by a discrete infinite group cannot be defined in an elementary way.
To deal with this, we first define a version of the fpqc site of $k$.
We denote it by $\pfpqc{k}$ and call it the perfect fpqc site of $k$ (Section \ref{Sec:FpqcMain}).
The underlying category of $\pfpqc{k}$ consists of
perfect $k$-schemes (perfect means that the Frobenius is invertible),
so that it contains quasi-algebraic groups.
Note that the perfect \'etale topology explained at
\cite[III, \S 0, ``Duality for unipotent perfect group schemes'']{Mil06}
is insufficient for our purpose,
since in general a surjection of proalgebraic groups is
not a surjection in the perfect \'etale topology
as it is not necessarily of finite presentation.
Also we have to be a bit careful about flatness,
since in general the relative Frobenius morphism of a $k$-scheme is not flat
and, unlike the perfect \'etale topology,
a scheme flat over a perfect $k$-scheme could be imperfect.
What we need for these points is Proposition \ref{Prop:Site} below.
The category $\psh{k}$ of sheaves of abelian groups on $\pfpqc{k}$
is the category we choose to work with.
Toward defining $\pi_{i}^{k}$ on $\psh{k}$,
the problem is that $\psh{k}$ no longer has sufficient projective objects,
so we cannot define $\pi_{i}^{k}$ to be the left derived functors of $\pi_{0}^{k}$.
Instead, we use the Ext functor for $\psh{k}$ to define $\pi_{i}^{k}$
(Section \ref{Sec:Homot}).
We prove in Section \ref{Sec:Thick} that $\psh{k}$ contains
both the category of commutative affine proalgebraic groups
and the category of commutative \'etale group schemes
both as abelian thick full subcategories.
The thickness implies that $\Ext_{k}^{1}$ and so $\pi_{1}^{k}$ are preserved.
To prove Theorem \ref{Th:Main},
we want to imitate Serre's proof of Theorem \ref{Th:SerreLCFT}.
This gives rise to two problems.
One problem is that the exactness of a sequence in $\psh{k}$
is not always determined by the exactness of the sequence induced on the groups of $\algcl{k}$-points.
We deal with this by defining Tate cohomology not as groups but as sheaves (Section \ref{Sec:Tate})
and giving one situation where $\algcl{k}$-points have enough information (Proposition \ref{Prop:Tate}).
Then we can convert the vanishing result \eqref{Eq:Vanish} of Tate cohomology groups
into that of Tate cohomology sheaves (Proposition \ref{Prop:Vanish}).
The other problem is that,
for a finite extension of complete discrete valuation fields  $L / K$ with residue extension $k' / k$,
we need to regard the group of units and the multiplicative group of $L$
as sheaves in several different ways,
some of which defined on $\pfpqc{k}$ and others on $\pfpqc{k'}$.
These are sheaf versions of the groups defined at \cite[XIII, \S 5, Exercise 2]{Ser79}.
To define these sheaves and make the discussion smooth,
we discuss a version of the Greenberg functor (\cite[V, \S 4, no.\ 1]{DG70}) in Section \ref{Sec:Green}.
All the above is the way we take here for Theorem \ref{Th:Main}
(and was for our original paper \cite{SY10}).
Of course this is not the only way to formulate and obtain the same theorem or its equivalent.
Several different approaches will be possible.
Nevertheless, the authors believe that the way we take here is at least one of the most standard ways.
Note that some part of the machinery in Sections \ref{Sec:Fpqc} and \ref{Sec:FormPf}
has already been used at \cite{Suz09}.

From now on throughout this paper,
we always mean by $k$ a fixed perfect field of characteristic $p > 0$.


\subsection{Definition and first properties} \label{Sec:FpqcMain}
We define the site $\pfpqc{k}$ below.
We first recall perfect rings and perfect schemes
(cf.\ \cite{Gre65}).

A $k$-algebra $R$ is called perfect if the $p$-th power map on $R$ is bijective.
Any perfect $k$-algebra $R$ is reduced,
since if $r^{n} = 0$ for $r \in R$ and $n \ge 1$,
then $r^{p^{e}} = r^{p^{e} - n} r^{n} = 0$ for $e \ge 0$ with $p^{e} > n$, so $r = 0$.%
	\footnote{This argument shows that
	the $p$-reducedness defined in \cite{Gre65} is the same as the usual reducedness.}
For a $k$-algebra $R$ and a non-negative integer $i$,
let $R^{(i)}$ be the $k$-algebra $R$ whose structure map is replaced by
the composite of the $p^{i}$-th power map $k \isomto k$ and the original structure map $k \to R$.
We denote by $R^{(\infty)}$ the perfect $k$-algebra
defined by the injective limit
	\[
		R \to R^{(1)} \to R^{(2)} \to \cdots,
	\]
where $R^{(i)} \to R^{(i + 1)}$ is the $p$-th power map.
Likewise, a $k$-scheme $X$ is perfect if the Frobenius morphism on $X$ is an isomorphism.
We denote by $\Perf{k}$ the full subcategory
of the category of $k$-schemes $\Sch{k}$ consisting of perfect $k$-schemes.
Perfectness is Zariski-local.
A perfect $k$-scheme is reduced.
For a $k$-scheme $X$ and a non-negative integer $i$,
we denote by $X^{(i)}$ the $k$-scheme $X$ whose structure morphism is replaced by
the composite of the original structure morphism $X \to \Spec k$
and the $p^{i}$-th power Frobenius morphism $\Spec k \isomto \Spec k$.
We denote by $X^{(\infty)}$ the perfect $k$-scheme defined by the projective limit
	\[
		\cdots \to X^{(2)} \to X^{(1)} \to X,
	\]
where $X^{(i + 1)} \to X^{(i)}$ is the Frobenius morphism.
The functor $(\infty) \colon \Sch{k} \to \Perf{k}$ sending $X \mapsto X^{(\infty)}$
is right adjoint to the inclusion functor $\Perf{k} \into \Sch{k}$.
The fiber product of two perfect $k$-schemes over a perfect $k$-scheme
taken in $\Sch{k}$ gives a perfect $k$-scheme.
A $k$-scheme (resp.\ $k$-algebra) is said to be quasi-algebraic
if it can be obtained by applying the functor $(\infty)$ to an algebraic (that is, of finite type) one.

Now we define a site $\pfpqc{k}$, which we call the perfect fpqc site of $k$, as follows.
The underlying category of $\pfpqc{k}$ is the category of perfect $k$-schemes $\Perf{k}$.
Its topology is the fpqc topology,
namely a family of morphisms $\{X_{\lambda} \to X\}$ in $\Perf{k}$
is a covering for $\pfpqc{k}$ if it is a covering for the fpqc site $\fpqc{k}$ of $k$
(cf.\ \cite[Exp.\ IV]{SGA3-1}, \cite{SGA4-1}).
We denote by $\psh{k}$ (resp.\ $\sh{k}$)
the category of sheaves of abelian groups on $\pfpqc{k}$ (resp.\ $\fpqc{k}$).

\begin{prp} \label{Prop:Site}
	The functor $(\infty) \colon \Sch{k} \to \Perf{k}$ gives a morphism of sites
	$\pfpqc{k} \to \fpqc{k}$.
	The pullback $(\infty)^{\ast} \colon \sh{k} \to \psh{k}$ is given by
	the restriction functor $|_{\Perf{k}}$.
	In particular, $|_{\Perf{k}}$ is an exact functor.
\end{prp}

\begin{proof}
	The only non-trivial part is that
	$(\infty) \colon \Sch{k} \to \Perf{k}$ sends
	covering families for $\fpqc{k}$ to those for $\pfpqc{k}$.
	It suffices to show that if $S$ is a flat algebra over a $k$-algebra $R$,
	then $S^{(\infty)}$ is flat over $R^{(\infty)}$.
	We have $S^{(i)} = S$ and $R^{(i)} = R$ for $i \ge 0$
	if we forget the $k$-algebra structures.
	Therefore $S^{(i)}$ is flat over $R^{(i)}$.
	Taking injective limits, we know that $S^{(\infty)}$ is flat over $R^{(\infty)}$.
\end{proof}

If $F \in \psh{k}$, then 
for a perfect $k$-algebra $R$ and a finite family of perfect $R$-algebras $R_{i}$
with $\prod R_{i}$ faithfully flat over $R$,
the sequence
	$
			F(R)
		\to
			\prod_{i} F(R_{i})
		\rightrightarrows
			\prod_{i, j} F(R_{i} \tensor_{R} R_{j})
	$
is exact.
Conversely, if a covariant functor $F$ from the category of perfect $k$-algebras
to the category of abelian groups satisfies this condition,
then the Zariski-sheafification of $F$ is in $\psh{k}$.
This correspondence sets up an equivalence of categories,
which follows from the corresponding fact for the usual fpqc site
(cf.\ for example \cite[Prop.\ 9.3 and Cor.\ 9.4]{Kre10})
and the fact that perfectness is Zariski-local.


\subsection{Homotopy groups and fundamental groups} \label{Sec:Homot}
Let $\Ext_{k}^{i}$ be the $i$-th Ext functor for $\psh{k}$.
For $i \ge 0$, we define the $i$-th homotopy group of $A \in \psh{k}$,
denoted by $\pi_{i}^{k}(A)$,
to be the Pontryagin dual of the torsion abelian group $\injlim_{n} \Ext_{k}^{i}(A, n^{-1} \Z / \Z)$.
We call $\pi_{1}^{k}(A)$ the fundamental group of $A$.
The system $\{\pi_{i}^{k}\}_{i \ge 0}$ is a covariant homological functor
from $\psh{k}$ to the category of profinite abelian groups.

\begin{prp} \label{Prop:Weil}
	Let $k' / k$ be a finite extension.
	We denote by $\Res_{k' / k}$ the Weil restriction functor
	$\psh{k'} \to \psh{k}$.
	\begin{enumerate}
		\item \label{Enum:WeilAdj}
			$\Res_{k' / k}$ is left adjoint to the restriction functor
			$\psh{k} \to \psh{k'}$.
			In particular, $\Res_{k' / k}$ is an exact functor.
		\item \label{Enum:HomotInv}
			We have a canonical isomorphism
			$\pi_{i}^{k}(\Res_{k' / k} F') \cong \pi_{i}^{k'}(F')$
			for $F' \in \psh{k'}$ and $i \ge 0$.
	\end{enumerate}
\end{prp}

\begin{proof}
	\ref{Enum:WeilAdj}.
	Let $F' \in \psh{k'}$ and $G \in \psh{k}$.
	They are sheaves also for the big \'etale site of perfect schemes
	respectively over $k'$ and over $k$.
	The Weil restriction functor is the pushforward functor
	by the finite \'etale morphism $\Spec k' \to \Spec k$.
	The claim can thus be proved by the same argument as
	the proof of \cite[V, \S 1, Lem.\ 1.12]{Mil80}.
	
	\ref{Enum:HomotInv}.
	Assertion \ref{Enum:WeilAdj} implies that
	$\Ext_{k}^{i}(\Res_{k' / k} F', G) \cong \Ext_{k'}^{i}(F', G)$ for all $i \ge 0$.
	Setting $G = n^{-1} \Z / \Z$, taking the injective limit in $n$ and taking the Pontryagin dual,
	we get the result.
\end{proof}


\subsection{Affine proalgebraic groups and \'etale group schemes} \label{Sec:Thick}
We show that $\psh{k}$ contains both the category of commutative affine proalgebraic groups
and the category of commutative \'etale group schemes
as abelian thick full subcategories.

\begin{prp} \label{Prop:ThickAff}
	The natural functor from the category of commutative affine proalgebraic groups over $k$
	to $\psh{k}$ is a fully faithful exact functor.
	Its essential image is the category of commutative perfect affine group schemes,
	which is thick (i.e.\ closed under extension) in $\psh{k}$.
\end{prp}

\begin{proof}
	Fully faithful.
	Let $A$, $B$ be commutative affine proalgebraic groups over $k$.
	By definition, $A$ (resp.\ $B$) can be written as the projective limit of
	some affine quasi-algebraic groups $A_{\lambda}$ (resp.\ $B_{\mu}$) over $k$.
	For an affine scheme $X$,
	we denote by $\Order(X)$ the ring of global sections of the structure sheaf of $X$.
	We denote by $\Hom_{\text{proalg}}$
	(resp.\ $\Hom_{\text{quasialg}}$, $\Hom_{\text{bialg}}$)
	the set of homomorphisms in the category of proalgebraic groups
	(resp.\ quasi-algebraic groups, bi-algebras) over $k$.
	We have
		\begin{align*}
					\Hom_{\text{proalg}}(A, B)
			&	=
					\projlim_{\mu} \,\, \injlim_{\lambda} \,\,
						\Hom_{\text{quasialg}}(A_{\lambda}, B_{\mu})
			\\
			&	=
					\projlim_{\mu} \,\, \injlim_{\lambda} \,\,
						\Hom_{\text{bialg}}(\Order(B_{\mu}), \Order(A_{\lambda}))
			\\
			&	=
					\projlim_{\mu} \,\,
						\Hom_{\text{bialg}}(\Order(B_{\mu}),  \injlim_{\lambda} \Order(A_{\lambda}))
			\\
			&	=
						\Hom_{\text{bialg}}(\injlim_{\mu} \Order(B_{\mu}),  \injlim_{\lambda} \Order(A_{\lambda}))
			\\
			&	=
					\Hom_{\text{bialg}}(\Order(B), \Order(A))
			\\
			&	=
					\Hom_{k}(A, B).
		\end{align*}
	(Here $\Hom_{k}$ is, as before, the set of homomorphisms in $\psh{k}$.)
	
	Essential image.
	Any affine group scheme over $k$ can be written
	as the projective limit of affine algebraic group schemes
	(\cite[III, \S 3, 7.5 Cor.\ (b)]{DG70}).
	This implies the result.
	
	Exact.
	It suffices to show that, for an injection of commutative affine proalgebraic groups $A \into B$,
	its cokernel in $\psh{k}$ is an affine proalgebraic group,
	or equivalently, a perfect affine group scheme.
	Since $A$ and $B$ are affine, we can naturally regard them as sheaves on $\fpqc{k}$.
	Let $C$ be the cokernel of $A \into B$ in $\sh{k}$.
	Since the restriction functor $|_{\Perf{k}} \colon \sh{k} \to \psh{k}$ is an exact functor
	by Proposition \ref{Prop:Site}, we know that
	$C|_{\Perf{k}}$ gives the cokernel of $A \into B$ in $\psh{k}$.
	The remaining task is to show that $C$ is representable by a perfect affine group scheme.
	By \cite[III, \S 3, 7.2 Th.]{DG70}, $C$ is representable by an affine group scheme.
	Consider the commutative diagram in $\sh{k}$ with exact rows
		\[
			\begin{CD}
					0 @>>> A^{(1)} @>>> B^{(1)} @>>> C^{(1)} @>>> 0
				\\
					@. @VVV @VVV @VVV @.
				\\
					0 @>>> A @>>> B @>>> C @>>> 0,
			\end{CD}
		\]
	where vertical arrows are given by the Frobenius morphisms.
	Since $A$ and $B$ are perfect,
	the first and second vertical morphisms are isomorphism,
	hence so is the third.
	Therefore $C$ is perfect.
	
	Thick.
	Let $0 \to A \to B \to C \to 0$ be an exact sequence in $\psh{k}$
	with $A$ and $C$ perfect affine.
	Then $B$ is an $A$-torsor over $C$ for the perfect fpqc topology.
	Therefore $B$ is perfect affine
	by the fpqc descent for affine morphisms
	together with an argument similar to the proof of \cite[III, \S 4, 1.9 Prop.\ (a)]{DG70}.
\end{proof}

\begin{prp} \label{Prop:ThickEt}
	Commutative \'etale group schemes over $k$
	form an abelian thick full subcategory of $\psh{k}$.
\end{prp}

\begin{proof}
	Any scheme \'etale over $k$ is a disjoint union of the $\Spec$'s of
	(an infinite number of) finite extensions of the perfect field $k$.
	Such a scheme is perfect.
	The proposition is obvious except the thickness.
	For the thickness, it is enough to show that
	a torsor $B$ over $\Spec k$ for the perfect fpqc topology
	under a commutative \'etale group scheme $A$ is representable by an \'etale scheme.
	Such a torsor can be trivialized by extending the base $\Spec k$
	to a perfect affine $k$-scheme $\Spec R$ faithfully flat over $\Spec k$
	(i.e.\ $R \ne 0$).
	It is enough to show that $R$ can be taken to be a finite Galois extension of $k$.
	
	Actually it is enough to show that $R$ can be taken to be
	quasi-algebraic (see Section \ref{Sec:FpqcMain} for the definition) over $k$
	because of the following argument.
	Let $\mathfrak{m}$ be a maximal ideal of $R \ne 0$.
	If $R$ is quasi-algebraic, then
	$R / \mathfrak{m}$ is a finite extension of $k$
	by the Noether normalization theorem.
	Take a finite Galois extension $k'$ of $k$ containing $R / \mathfrak{m}$.
	Then we can replace $R$ by $k'$.
	
	Now we show that $R$ can be taken to be quasi-algebraic over $k$.
	Recall from \cite[III, \S 4, 6.5]{DG70} that
	there is an isomorhism between the group of $A$-torsors over $\Spec k$ that can be trivialized by $\Spec R$
	and the first Amitsur cohomology group $H_{\mathrm{Am}}^{1}(R / k, A)$ of $R / k$ with coefficients in $A$.
	Therefore it is enough to show that there exists a quasi-algebraic $k$-subalgebra $R_{1}$ of $R$ such that
	the Amitsur cocycle class $\bar{\sigma} \in H_{\mathrm{Am}}^{1}(R / k, A)$ that corresponds to the torsor $B$
	belongs to the subgroup $H_{\mathrm{Am}}^{1}(R_{1} / k, A)$.
	Let $\sigma \colon \Spec R \times_{k} \Spec R \to A$ be a representative of the cocycle class $\bar{\sigma}$.
	It is enough to show that there exists a quasi-algebraic $k$-subalgebra $R_{1}$ of $R$ such that
	$\sigma$ factors as the natural morphism
	$\Spec R \times_{k} \Spec R \to \Spec R_{1} \times_{k} \Spec R_{1}$
	followed by some cocycle $\sigma_{1} \colon \Spec R_{1} \times_{k} \Spec R_{1} \to A$.
	We first construct such $\sigma_{1}$ just as a morphism of schemes and then show that it is indeed a cocycle.
	
	Since $\Spec R \times_{k} \Spec R = \Spec(R \tensor_{k} R)$ is quasi-compact
	and $A$ is \'etale over $k$,
	the morphism $\sigma \colon \Spec R \times_{k} \Spec R \to A$ factors through
	a finite subscheme $\Spec(k_{1} \times \dots \times k_{n})$ of $A$,
	where the $k_{i}$ are finite extensions of $k$.
	The corresponding $k$-algebra homomorphism
	$k_{1} \times \dots \times k_{n} \to R \tensor_{k} R$
	factors through $R_{1} \tensor_{k} R_{1}$,
	where $R_{1}$ is a quasi-algebraic $k$-subalgebra of $R$.
	Therefore the morphism $\sigma \colon \Spec R \times_{k} \Spec R \to A$ factors as
	$\Spec R \times_{k} \Spec R \to \Spec R_{1} \times_{k} \Spec R_{1}$
	followed by $\Spec R_{1} \times_{k} \Spec R_{1} \to A$.
	We denote this morphism $\Spec R_{1} \times_{k} \Spec R_{1} \to A$ by $\sigma_{1}$.
	
	Finally we show that the morphism $\sigma_{1} \colon \Spec R_{1} \times_{k} \Spec R_{1} \to A$ is a cocycle,
	namely its coboundary
		$
				\partial \sigma_{1}
			\colon
				\Spec R_{1} \times_{k} \Spec R_{1} \times_{k} \Spec R_{1}
			\to
				A
		$
	is zero.
	The composite of the natural morphism
		$
				\Spec R \times_{k} \Spec R \times_{k} \Spec R
			\to
				\Spec R_{1} \times_{k} \Spec R_{1} \times_{k} \Spec R_{1}
		$
	and $\partial \sigma_{1}$ is $\partial \sigma$,
	which is zero since $\sigma$ is a cocycle.
	This implies that $\partial \sigma_{1}$ is zero since
		$
				R_{1} \tensor_{k} R_{1} \tensor_{k} R_{1}
			\to
				R \tensor_{k} R \tensor_{k} R
		$
	is injective.
	Therefore $\sigma_{1}$ is a cocycle.
	This completes the proof.
	
\end{proof}

These propositions have several consequences.
By Proposition \ref{Prop:ThickAff},
we may identify the category of commutative affine proalgebraic groups
with the category of commutative perfect affine group schemes,
which is the quotient category of the category of commutative affine group schemes
by its full subcategory of proinfinitesimal groups.

For commutative affine proalgebraic groups $A$ and $B$,
the group of first extension classes of $B$ by $A$ as commutative proalgebraic groups
is the same as that as sheaves of abelian groups on $\pfpqc{k}$ and on $\fpqc{k}$
by Proposition \ref{Prop:ThickAff}.
In particular, our $\pi_{1}^{k}(A)$ for a commutative proalgebraic group $A$
defined in Section \ref{Sec:Homot} coincides with Hazewinkel's $\pi_{1}^{k}(A)$.
For a commutative \'etale group $A$ over $k$,
the first cohomology group of $\pfpqc{k}$ with values in $A$
is equal to the group of first extension classes of $\Z$ by $A$ in $\pfpqc{k}$,
which is equal to that in the \'etale site of $k$ by Proposition \ref{Prop:ThickEt},
which in turn is equal to the Galois cohomology group
$H^{1}(k, A) = H^{1}(\Gal(\algcl{k} / k), A(\algcl{k}))$.%
	\footnote{Do not confuse this type of Galois cohomology groups
	with Tate cohomology sheaves that we define in the next section.}
In particular, we have
$\injlim_{n} \Ext_{k}^{1}(\Z, n^{-1} \Z / \Z) \cong H^{1}(\Gal(\algcl{k} / k), \Q / \Z)$,%
	\footnote{$\Ext_{k}^{1}$ here and $\Hom_{k}$ below are relative to the site $\pfpqc{k}$ as before.}
	which shows that $\pi_{1}^{k}(\Z) \cong \Gal(k^{\ab} / k)$.

We define two homomorphisms, \eqref{Eq:HomOnRat} and \eqref{Eq:HomOnRatF} below,
that are related to $\pi_{1}^{k}$
and will be used later in Section \ref{Sect:Aux}.
For a sheaf $A \in \psh{k}$,
we have a natural homomorphism
	\begin{equation} \label{Eq:HomOnRat}
			A(k)
		=
			\Hom_{k}(\Z, A)
		\overset{\pi_{1}^{k}}{\to}
			\Hom(\pi_{1}^{k}(\Z), \pi_{1}^{k}(A))
		\cong
			\Hom(\Gal(k^{\ab} / k), \pi_{1}^{k}(A)).
	\end{equation}
More explicitly,
this comes from the homomorphism
$\Ext_{k}^{1}(A, N) \to \Hom(A(k), H^{1}(k, N))$ for finite constant $N$
that sends an extension class $0 \to N \to B \to A \to 0$
to the coboundary map of Galois cohomology $A(k) = H^{0}(k, A) \to H^{1}(k, N)$.
If $k$ is quasi-finite in the sense of \cite[XIII, \S 2]{Ser79},
then $\Gal(\algcl{k} / k) \cong \hat{\Z}$,
so \eqref{Eq:HomOnRat} is reduced to a homomorphism
	\begin{equation} \label{Eq:HomOnRatF}
		A(k) \to \pi_{1}^{k}(A).
	\end{equation}

If we further assume that $A$ is a connected affine proalgebraic group over $k$,
then the homomorphism \eqref{Eq:HomOnRatF} above and hence $\pi_{1}^{k}(A)$ can be understood nearly completely
by the following proposition.
This result will not be used later.

\begin{prp}
	Assume $k$ is quasi-finite.
	Let $A$ be a connected affine proalgebraic group over $k$.
	\begin{enumerate}
		\item \label{Enum:PiOne:QFin}
			The homomorphism \eqref{Eq:HomOnRatF} induces an isomorphism
			from the completion of $A(k)$ by its normic subgroups
			(\cite[XV, \S 1, Exercise 2]{Ser79}) to $\pi_{1}^{k}(A)$.
		\item \label{Enum:PiOne:Fin}
			If $k$ is a finite field with $q$ elements,
			then the homomorphism \eqref{Eq:HomOnRatF} is an isomorphism.
			Its inverse is given by the boundary map of the homotopy long exact sequence
			for Lang's short exact sequence
			$0 \to A(k) \to A \overset{F - 1}{\to} A \to 0$
			(\cite[III, \S 5, 7.2]{DG70}),
			where $F$ is the $q$-th power Frobenius morphism.
	\end{enumerate}
\end{prp}

\begin{proof}
	\ref{Enum:PiOne:QFin}.
	It is enough to show that for finite constant $N$,
	the map $\Ext_{k}^{1}(A, N) \to \Hom(A(k), H^{1}(k, N)) = \Hom(A(k), N)$ is an injection
	whose image consists of all homomorphisms to $N$ with normic kernels.
	For the injectivity, let $0 \to N \to B \to A \to 0$ be an extension class
	whose image in $\Hom(A(k), N)$ is trivial,
	namely the coboundary map of Galois cohomology $A(k) \to H^{1}(k, N) = N$ is zero.
	Then the homomorphism $N = H^{1}(k, N) \to H^{1}(k, B)$ is injective.
	This is surjective since $H^{1}(k, A) = 0$ by \cite[XV, \S 1, Exercise 2 (a)]{Ser79}.
	Let $B_{0}$ be the connected component of $B$ containing the identity.
	Then we have $H^{1}(k, B) \cong H^{1}(k, B / B_{0})$
	since $H^{i}(k, B_{0}) = 0$ for any $i \ge 1$ by the same exercise.
	The group $B / B_{0}$ is pro-finite-\'etale.
	Hence $H^{1}(k, B / B_{0})$ is the cokernel of the endomorphism $F - 1$ on $B / B_{0}$
	by \cite[XIII, \S 1, Prop.\ 1]{Ser79},
	where $F$ is the given generator of $\Gal(\algcl{k} / k) \cong \hat{\Z}$.
	This cokernel is $\pi_{0}^{k}(B)$ by definition.
	The isomorphism $N \cong \pi_{0}^{k}(B)$ thus obtained is given by the composite of
	the natural homomorphisms $N \into B \onto \pi_{0}^{k}(B)$.
	Therefore $N$ is a direct factor of $B$,
	so the extension class $0 \to N \to B \to A \to 0$ is trivial.
	This shows the injectivity of $\Ext_{k}^{1}(A, N) \to \Hom(A(k), N)$.
	The characterization of the image of this homomorphism
	directly follows from the definition of normic subgroups.
	
	\ref{Enum:PiOne:Fin}.
	First we show that the endomorphism $F - 1$ induces zero maps on homotopy groups.
	We have an equality $\Hom_{k}(F, \id) = \Hom_{k}(\id, F)$
	as endomorphisms of the abelian group $\Hom_{k}(A, N)$ for any sheaf $N \in \psh{k}$.
	Hence we have $\Ext_{k}^{i}(F, \id) = \Ext_{k}^{i}(\id, F)$
	as endomorphisms of $\Ext_{k}^{i}(A, N)$ for any $i \ge 0$.
	If $N$ is constant, we have $F = \id$ on $N$, so $\Ext_{k}^{i}(\id, F) = \id$.
	Therefore $F = \id$ ($= 1$) and $F - 1 = 0$ on $\pi_{i}^{k}(A)$.
	
	Therefore the boundary map of the homotopy long exact sequence
	for $0 \to A(k) \to A \overset{F - 1}{\to} A \to 0$ gives an isomorphism
	$\pi_{1}^{k}(A) \isomto \pi_{0}^{k}(A(k)) = A(k)$%
		\footnote{This means that $F - 1$ is universal among isogenies onto A with proconstant kernels.}
	since $A$ is connected.
	We show that the composite $A(k) \to \pi_{1}^{k}(A) \isomto A(k)$,
	or the composite $\Hom(A(k), N) \isomto \Ext_{k}^{1}(A, N) \to \Hom(A(k), N)$
	for any (pro-)finite constant $N$, is the identity map.
	It is enough to show that, taking $N = A(k)$,
	the composite map
	$\Hom(A(k), A(k)) \to \Ext_{k}^{1}(A, A(k)) \to \Hom(A(k), A(k))$
	sends the identity map to itself.
	The identity map in $\Hom(A(k), A(k))$ goes to the extension class
	$0 \to A(k) \to A \overset{F - 1}{\to} A \to 0$ in $\Ext_{k}^{1}(A, A(k))$.
	The coboundary map of Galois cohomology for this sequence is the identity map.
	Hence we get the result.
\end{proof}


\subsection{Tate cohomology sheaves} \label{Sec:Tate}
Let $G$ be a finite (abstract) group
and let $A$ be a sheaf of $G$-modules on $\pfpqc{k}$.
For each $i \in \Z$, we define the $i$-th Tate cohomology sheaf of $G$ with values in $A$,
denoted by $\hat{H}^{i}(G, A)$, as follows.
For $i \ge 0$, let $C^{i}(G, A)$ be the product of copies of $A$
labeled by the finite set $G^{i} = G \times \dots \times G$.
For $i < 0$, let $C^{i}(G, A)$ be the product of copies of $A$ labeled by $G^{- i + 1}$.
We can define differentials $\{d_{i}\}_{i \in \Z}$ for $\{C^{i}(G, A)\}_{i \in \Z}$ as morphisms of sheaves
by using the differentials for the standard complete complex of the usual Tate cohomology
in inhomogeneous cochain presentation,
namely for $i \ge 0$ (resp.\ $i < -1$),
we use the formula in \cite[VII, \S 3]{Ser79} (resp.\ \cite[VII, \S 4]{Ser79})
and for $i = -1$, we use the norm map $N_{G} = \sum_{\sigma \in G} \sigma \colon A \to A$.
We then define $\hat{H}^{i}(G, A)$ to be the $i$-th cohomology of the complex $\{(C^{i}(G, A), d_{i})\}_{i \in \Z}$,
namely $\Ker(d_{i}) / \Im(d_{i - 1})$, the image and the quotient being taken in $\psh{k}$.%
	\footnote{This is different from the Galois cohomology group
	$H^{i}(k, A) = H^{i}(\Gal(\algcl{k} / k), A(\algcl{k}))$.
	The group $A(\algcl{k})$ has actions of both groups $G$ and $\Gal(\algcl{k} / k)$ (commuting with each other),
	the first one coming from the action of $G$ on the sheaf $A$,
	the second from the action of $\Gal(\algcl{k} / k)$ on the coefficient field $\algcl{k}$.
	These actions are unrelated in general,
	so are the corresponding cohomology theories.}
The sheaf $\hat{H}^{i}(G, A)$ is the sheafification of the presheaf
$R \mapsto \hat{H}^{i}(G, A(R))$ of the usual Tate cohomology group of $G$
with values in the $G$-module $A(R)$.
A short exact sequence of sheaves of $G$-modules on $\pfpqc{k}$
induces a long exact sequence of the Tate cohomology sheaves.

\begin{prp} \label{Prop:Tate}
	Let $G$ be a finite group and let $A$ be a sheaf of $G$-modules on $\pfpqc{k}$.
	We assume the following three conditions.
		\begin{itemize}
			\item
				There exists an exact sequence
				$0 \to A_{a} \to A \to A_{e} \to 0$
				of sheaves of $G$-modules on $\pfpqc{k}$.
			\item
				$A_{a}$ is an affine proalgebraic group.
			\item
				$A_{e}$ is an \'etale group scheme with
				$A_{e}(\algcl{k})$ finitely generated as an abelian group.
		\end{itemize}
	Then the $i$-th Tate cohomology sheaf $\hat{H}^{i}(G, A)$ is affine for each $i \in \Z$.
	The group of its $\algcl{k}$-points is given by the Tate cohomology group
	$\hat{H}^{i}(G, A(\algcl{k}))$.
\end{prp}

\begin{proof}
	First we show that $\hat{H}^{i}(G, A_{a})$ is affine
	and $\hat{H}^{i}(G, A_{a})(\algcl{k}) = \hat{H}^{i}(G, A_{a}(\algcl{k}))$.
	The complex $\{C^{i}(G, A_{a})\}_{i \in \Z}$ consists of affine proalgebraic groups.
	Therefore its cohomology $\hat{H}^{i}(G, A_{a})$ is affine.
	Since the functor sending a commutative proalgebraic group
	to the group of its $\algcl{k}$-points is exact,
	we know that $\hat{H}^{i}(G, A_{a})(\algcl{k}) = \hat{H}^{i}(G, A_{a}(\algcl{k}))$.
	
	The same argument shows that $\hat{H}^{i}(G, A_{e})$ is an \'etale group scheme.
	Moreover we know that $\hat{H}^{i}(G, A_{e})$ is finite
	since $A_{e}(\algcl{k})$ is a finitely generated abelian group
	(\cite[VIII, \S 2, Cor.\ 2]{Ser79}).
	In particular, $\hat{H}^{i}(G, A_{e})$ is affine.
	Since the functor sending a commutative \'etale group scheme
	to the group of its $\algcl{k}$-points is exact,
	we know that $\hat{H}^{i}(G, A_{e})(\algcl{k}) = \hat{H}^{i}(G, A_{e}(\algcl{k}))$.
	
	Now we show that the sheaf $\hat{H}^{i}(G, A)$ is affine.
	The short exact sequence
	$0 \to A_{a} \to A \to A_{e} \to 0$
	induces a long exact sequence
		\begin{equation} \label{Eq:SeqCoh}
				\cdots
			\to
				\hat{H}^{i - 1}(G, A_{e})
			\overset{d_{i - 1}}{\to}
				\hat{H}^{i}(G, A_{a})
			\to
				\hat{H}^{i}(G, A)
			\to
				\hat{H}^{i}(G, A_{e})
			\overset{d_{i}}{\to}
				\hat{H}^{i + 1}(G, A_{a})
			\to
				\cdots
		\end{equation}
	and a short exact sequence
		\[
				0
			\to
				\Coker(d_{i - 1})
			\to
				\hat{H}^{i}(G, A)
			\to
				\Ker(d_{i})
			\to
				0.
		\]
	As we saw above, the domain and the codomain of the morphism
	$d_{i - 1} \colon \hat{H}^{i - 1}(G, A_{e}) \to \hat{H}^{i}(G, A_{a})$
	are affine.
	Hence so are $\Coker(d_{i - 1})$ and $\Ker(d_{i})$.
	By the thickness of the category of affine proalgebraic groups in $\psh{k}$
	(Proposition \ref{Prop:ThickAff}),
	the sheaf $\hat{H}^{i}(G, A)$ is affine.
	
	Next we show that $\hat{H}^{i}(G, A)(\algcl{k}) = \hat{H}^{i}(G, A(\algcl{k}))$.
	Since each term of the sequence \eqref{Eq:SeqCoh} is an affine proalgebraic group,
	the corresponding sequence for $\algcl{k}$-points
		\begin{equation} \label{Eq:SeqCohRat}
				\cdots
			\to
				\hat{H}^{i - 1}(G, A_{e})(\algcl{k})
			\to
				\hat{H}^{i}(G, A_{a})(\algcl{k})
			\to
				\hat{H}^{i}(G, A)(\algcl{k})
			\to
				\hat{H}^{i}(G, A_{e})(\algcl{k})
			\to
				\hat{H}^{i + 1}(G, A_{a})(\algcl{k})
			\to
				\cdots
		\end{equation}
	is exact.
	On the other hand,
	the sequence $0 \to A_{a}(\algcl{k}) \to A(\algcl{k}) \to A_{e}(\algcl{k}) \to 0$ is exact,
	since fibers of the morphism $A \onto A_{e}$ over $\algcl{k}$-points of $A_{e}$ are
	$A_{a}$-torsors over $\Spec \algcl{k}$ for the perfect fpqc topology,
	which have to be trivial by Lemma \ref{Lem:TrivTorsor} below.
	Consider the resulting long exact sequence
		\begin{equation} \label{Eq:SeqRatCoh}
				\cdots
			\to
				\hat{H}^{i - 1}(G, A_{e}(\algcl{k}))
			\to
				\hat{H}^{i}(G, A_{a}(\algcl{k}))
			\to
				\hat{H}^{i}(G, A(\algcl{k}))
			\to
				\hat{H}^{i}(G, A_{e}(\algcl{k}))
			\to
				\hat{H}^{i + 1}(G, A_{a}(\algcl{k}))
			\to
				\cdots.
		\end{equation}
	The terms in the sequences \eqref{Eq:SeqCohRat} and \eqref{Eq:SeqRatCoh}
	except the middle ones are isomorphic.
	Hence so are the middle.
\end{proof}

The following lemma used above should be well-known at least in the case of the usual fpqc topology.
But the authors could not find an appropriate reference.
Let us give a proof of it.

\begin{lmm} \label{Lem:TrivTorsor}
	Assume that $k$ is algebraically closed.
	Let $A$ be a commutative affine proalgebraic group over $k$.
	Then any $A$-torsor over $\Spec k$ for the perfect fpqc topology is trivial.
\end{lmm}

\begin{proof}
	Let $X$ be such a torsor.
	As in the proof of the part of Proposition \ref{Prop:ThickAff} for thickness,
	$X$ is representable by a perfect affine scheme.
	Let $T$ be the set of pairs $(B, x)$,
	where $B$ is a proalgebraic subgroup of $A$
	and $x$ is a $k$-point of the quotient $A / B$-torsor $X / B$.
	The set $T$ is non-empty since $X / A = \Spec k$.
	Also $T$ has a natural order.
	We want to show that $T$ contains a pair $(B, x)$ with $B = 0$.
	
	We first show that $T$ contains a minimal element.
	Let $\{(B_{\lambda}, x_{\lambda})\}$ be a totally ordered sequence in $T$.
	We set $B = \bigcap B_{\lambda}$.
	The natural $A$-morphism $X / B \to \projlim X / B_{\lambda}$ is an isomorphism
	since the both sides are $A / B$-torsors.
	Let $x = (x_{\lambda}) \in \projlim (X / B_{\lambda})(k) \cong (X / B)(k)$.
	Then the pair $(B, x)$ is a lower bound of $\{(B_{\lambda}, x_{\lambda})\}$.
	By Zorn's lemma, we know that $T$ contains a minimal element.
	
	Let $(B, x)$ be a minimal element of $T$.
	We show that $B = 0$.
	Since $A$ is proalgebraic,
	it is enough to see that
	any proalgebraic subgroup $C \subset A$ with $A / C$ quasi-algebraic contains $B$.
	The fiber $F$ of the projection $X / (B \cap C) \onto X / B$ over $x \in (X / B)(k)$
	is a $B / B \cap C$-torsor.
	Since $B / B \cap C = (B + C) / C \subset A / C$ and $A / C$ is quasi-algebraic,
	we know that $B / B \cap C$ is quasi-algebraic as well.
	Hence the $B / B \cap C$-torsor $F$ is quasi-algebraic
	by \cite[I, \S 3, 1.11 Prop.]{DG70}.
	Therefore $F$ has a $k$-point $y$ by the Noether normalization theorem.
	The pair $(B \cap C, y)$ is an element of $T$
	that is less than or equal to $(B, x)$.
	By minimality, we have $(B \cap C, y) = (B, x)$, so $B \subset C$.
\end{proof}


\subsection{The perfect Greenberg functor} \label{Sec:Green}
We quickly recall the Greenberg functor over $k$ (cf.\ \cite[V, \S 4, no.\ 1]{DG70}).
Let $W$ be the ring scheme of Witt vectors of infinite length over $k$.
A profinite $W(k)$-module, defined at \cite[V, \S 2, 1.1]{DG70},
is a pro-object in the category of $W(k)$-modules of finite length.
The functor $\alg{M} \mapsto \alg{M}(k)$ from the category of affine $W$-modules
to the category of profinite $W(k)$-modules admits a left adjoint,
called the Greenberg functor.
The Greenberg functor induces an equivalence of categories
from the category of profinite $W(k)$-modules
to the quotient category of the category of affine $W$-modules by the subcategory of proinfinitesimal ones (\cite[V, \S 4, 1.8 Rem.\ (a)]{DG70}).

As mentioned at the end of Section \ref{Sec:Thick},
the category of commutative perfect affine group schemes over $k$
is the quotient category of commutative affine group schemes
by its full subcategory of proinfinitesimal groups.
Therefore the composite of the Greenberg functor and the functor $(\infty)$
gives an equivalence of categories from the category of profinite $W(k)$-modules
to the category of perfect affine $W^{(\infty)}$-modules.
We call this composite functor the perfect Greenberg functor over $k$
and denote it by $\Grn_{k}$.
Its inverse is given by the functor $\alg{M} \mapsto \alg{M}(k)$.
More explicit description of $\Grn_{k}$ is given as follows.

\begin{prp}
	For a profinite $W(k)$-module $M$ and a perfect $k$-algebra $R$, the natural map
	$W(R) \ctensor_{W(k)} M \to (\Grn_{k} M)(R)$
	is an isomorphism, where $\ctensor$ denotes the completed tensor product.
\end{prp}

\begin{proof}
	By \cite[V, \S 4, 1.7 Rem.\ (b)]{DG70},
	it is enough to see that the natural map
	$W(R) \ctensor_{W(k)} M \to W(S) \ctensor_{W(k)} M$
	is injective for any faithfully flat $R$-algebra $S$.
	We have $R^{(\infty)} = R$ since $R$ is perfect.
	Therefore the $R$-algebra $S^{(\infty)}$ is faithfully flat
	as shown in the proof of Proposition \ref{Prop:Site}.
	Replacing $S$ by $S^{(\infty)}$,
	we may assume $S$ is perfect.
	Thus the problem is reduced to showing that
	$W_{n}(S)$ is faithfully flat over $W_{n}(R)$ for any $n \ge 0$
	if $R$ is a perfect $k$-algebra and $S$ is a faithfully flat perfect $R$-algebra.
	Note that $W_{n}(R) / p W_{n}(R) = R$ and $W_{n}(S) / p W_{n}(S) = S$ since $R$ and $S$ are perfect.
	Therefore the result follows from the local criterion of flatness.
\end{proof}

We will need the following proposition on the relation
between $\Grn_{k}$ and $\Grn_{k'}$, where $k' / k$ is a finite extension,
in terms of the Weil restriction $\Res_{k' / k}$.

\begin{prp} \label{Prop:GrnRes}
	Let $k' / k$ be a finite extension and let $A$ be a profinite $W(k')$-module.
	Then we have a canonical isomorphism
	$\Res_{k' / k} \Grn_{k'} A \cong \Grn_{k} A$,
	where we regard $A$ as a profinite $W(k)$-module to define $\Grn_{k} A$.
\end{prp}

\begin{proof}
	Since $\Grn_{k'} A$ is perfect affine, so is $\Res_{k' / k} \Grn_{k'} A$
	by \cite[I, \S 1, 6.6 Prop.\ (a)]{DG70}.
	Since $\Grn_{k'} A$ is a $W^{(\infty)}$-module over $k$,
	so is $\Res_{k' / k} \Grn_{k'} A$ over $k'$.
	Therefore both $\Res_{k' / k} \Grn_{k'} A$ and $\Grn_{k} A$
	are perfect affine $W^{(\infty)}$-modules over $k$.
	We have $(\Res_{k' / k} \Grn_{k'} A)(k) = A = (\Grn_{k} A)(k)$.
	This proves the result.
\end{proof}


\section{Local class field theory: a refinement}
\label{Sec:FormPf}
Let $K$ be a complete discrete valuation field with residue field $k$.
We denote by $\Order_{K}$ the ring of integers,
by $U_{K}$ the group of units,
by $U_{K}^{n}$ the group of $n$-th principal units and
by $\maxid_{K}$ the maximal ideal.
The ring of integers and the group of units
of the completion of the maximal unramified extension $\hat{K}^{\ur}$
is denoted by $\hat{\Order}_{K}^{\ur}$ and $\hat{U}_{K}^{\ur}$.


\subsection{Sheaves associated with a local field} \label{Sec:SheavesK}
As in \cite[V, \S 4, no.\ 3.1]{DG70},
we define $\alg{O}_{K} = \Grn_{k}(\Order_{K})$
(note that in the equal characteristic case,
we view $\Order_{K}$ as a profinite $W(k)$-module via $W(k) \onto k \into \Order_{K}$).
This has a natural $W^{(\infty)}$-algebra structure (\cite[V, \S 4, 2.6 Prop.]{DG70}).
We define a proalgebraic group $\alg{U}_{K}$ to be $\alg{O}_{K}^{\times}$.
For $n \ge 0$, we define a proalgebraic ideal
$\alg{p}_{K}^{n} \subset \alg{O}_{K}$ to be $\Grn_{k} \maxid_{K}^{n}$
and a proalgebraic subgroup $\alg{U}_{K}^{n} \subset \alg{U}_{K}$ to be $1 + \alg{p}_{K}^{n}$ if $n \ne 0$
(for $n = 0$, we set $\alg{U}_{K}^{0} = \alg{U}_{K}$).

We have the Teichm\"{u}ller lifting map
$\omega \colon \Ga^{(\infty)}\into W^{(\infty)} \to \alg{O}_{K}$.
If $\pi_{K}$ is a prime element of $\Order_{K}$ and $R$ is a perfect $k$-algebra,
every element of $\alg{O}_{K}(R)$ can be written as
$\sum_{n = 0}^{\infty} \omega(a_{n}) \pi_{K}^{n}$
for a unique sequence of elements $a_{n} \in R$.
In particular, $\pi_{K}$ is not a zero-divisor in $\alg{O}_{K}(R)$.

For a perfect $k$-algebra $R$, we define a ring $\alg{K}(R)$
to be $\alg{O}_{K}(R) \tensor_{\Order_{K}} K$
($= \alg{O}_{K}(R)[\pi_{K}^{-1}]$).
We show that $\alg{K}$ gives a sheaf of rings on $\pfpqc{k}$
in the manner described at the end of Section \ref{Sec:FpqcMain}.
Let $R$ be a perfect $k$-algebra and
let $R_{1}, \dots, R_{n}$ be perfect $R$-algebras with $\prod R_{i}$ faithfully flat over $R$.
The sheaf condition for $\alg{O}_{K}$ says that the sequence
	\[
			\alg{O}_{K}(R)
		\to
			\prod_{i} \alg{O}_{K}(R_{i})
		\rightrightarrows
			\prod_{i, j} \alg{O}_{K}(R_{i} \tensor_{R} R_{j})
	\]
is exact.
Since $K$ is flat over $\Order_{K}$ and the products are finite products, the sequence
	\[
			\alg{O}_{K}(R) \tensor_{\Order_{K}} K
		\to
			\prod_{i} \alg{O}_{K}(R_{i}) \tensor_{\Order_{K}} K
		\rightrightarrows
			\prod_{i, j} \alg{O}_{K}(R_{i} \tensor_{R} R_{j}) \tensor_{\Order_{K}} K
	\]
is exact.
Therefore $\alg{K}$ gives a sheaf on $\pfpqc{k}$.

Also the functor $\alg{K}^{\times}$ is a sheaf since we have a cartesian diagram
	\[
		\begin{CD}
				\alg{K}^{\times}
			@>>>
				\{1\} \cong \Spec k
			\\
			@VVV @VVV
			\\
				\alg{K} \times_{k} \alg{K}
			@>>>
				\alg{K},
		\end{CD}
	\]
where the bottom arrow is the multiplication
and the left arrow is given by $a \mapsto (a, a^{-1})$.
We have a natural morphism of sheaves of rings $\alg{O}_{K} \to \alg{K}$
and a natural morphism of sheaves of groups $\alg{U}_{K} \to \alg{K}^{\times}$.
These are injective since $\pi_{K}$ is not a zero-divisor in $\alg{O}_{K}(R)$
as we saw before.
We have $\alg{K}(k) = K$ and $\alg{K}(\algcl{k}) = \hat{K}^{\ur}$.
In general for a perfect field $k'$ containing $k$,
the ring $\alg{K}(k')$ is a complete discrete valuation field
whose normalized valuation is the lift of that for $K$.
If $\pi_{K}$ is a prime element of $\Order_{K}$ and $R$ is a perfect $k$-algebra,
every element of $\alg{K}(R)$ can be written as
$\sum_{n \in \Z} \omega(a_{n}) \pi_{K}^{n}$
for a unique sequence of elements $a_{n} \in R$ with $a_{n} = 0$ for $n < 0$ sufficiently small.

We define the valuation map as a morphism of sheaves $\alg{K}^{\times} \to \Z$ as follows.
For a perfect $k$-algebra $R$ and $x \in \Spec R$,
we denote by $k_{R, x}$ the residue field of $\Spec R$ at $x$.
The image of $a \in R$ by the natural $k$-algebra homomorphism $R \to k_{R, x}$
is denoted by $a(x)$.
For each $x \in \Spec R$, let $v_{x}$ be the composite
of the map $\alg{K}^{\times}(R) \to \alg{K}^{\times}(k_{R, x})$ coming from $R \to k_{R, x}$
and the map $\alg{K}^{\times}(k_{R, x}) \onto \Z$
coming from the normalized valuation of the complete discrete valuation field $\alg{K}(k_{R, x})$.

\begin{prp}
	\begin{enumerate}
		\item \label{Enum:Val:LocConst}
			For a perfect $k$-algebra $R$ and $\alpha \in \alg{K}^{\times}(R)$,
			the map $x \mapsto v_{x}(\alpha)$ from the underlying topological space of $\Spec R$
			to $\Z$ is locally constant.
			This defines a morphism of sheaves $\alg{K}^{\times} \to \Z$.
		\item \label{Enum:Val:Ex}
			The sequence $0 \to \alg{U}_{K} \to \alg{K}^{\times} \to \Z \to 0$
			is a split exact sequence in $\psh{k}$.
	\end{enumerate}
\end{prp}

\begin{proof}
	\ref{Enum:Val:LocConst}.
	We fix a prime element $\pi_{K}$ of $\Order_{K}$.
	Let $\alpha = \sum_{n \in \Z} \omega(a_{n}) \pi_{K}^{n}$
	with $a_{n} \in R$, $a_{n} = 0$ for $n < 0$ sufficiently small.
	For an integer $l$, we have
		\[
				\{
						x \in \Spec R
					\,|\,
						v_{x}(\alpha) \ge l
				\}
			=
				\{
						x \in \Spec R
					\,|\,
						a_{l - 1}(x) = a_{l - 2}(x) = \dots = 0
				\},
		\]
	which is a closed subset of $\Spec R$.
	We know this set is open as well by writing it as $\{x \in \Spec R \,|\, v_{x}(\alpha^{-1}) \le - l\}$.
	Therefore this set is open and closed.
	This proves Assertion \ref{Enum:Val:LocConst}.
	
	\ref{Enum:Val:Ex}.
	The injectivity of $\alg{U}_{K} \to \alg{K}^{\times}$ was proved before.
	The morphism $\alg{K}^{\times} \to \Z$ has a section
	corresponding to a prime element in $\alg{K}(k) = K$.
	We prove that the kernel of $\alg{K}^{\times} \to \Z$ is $\alg{U}_{K}$.
	An element $\alpha = \sum \omega(a_{n}) \pi_{K}^{n} \in \alg{K}^{\times}(R)$
	is in the kernel of $\alg{K}^{\times}(R) \to \Z(R)$ if and only if
	$a_{n}(x) = 0$ and $a_{0}(x) \ne 0$ for any $x \in \Spec R$ and $n < 0$.
	Since $R$ is reduced, this is equivalent to saying that
	$a_{n} = 0$ for $n < 0$ and $a_{0} \in R^{\times}$,
	which in turn is equivalent to $\alpha \in \alg{U}_{K}(R)$.
\end{proof}


\subsection{Sheaves associated with a finite extension of a local field} \label{Sec:SheavesL}
Let $L$ be a finite extension of $K$ with residue field $k'$.
The above constructions of sheaves can be made also for the pair $(L, k')$ instead of the pair $(K, k)$.
We write the resulting sheaves by $\alg{O}_{L, k'}, \alg{U}_{L, k'}, \alg{L}_{k'}$, etc.
For example, we regard the ring of integers $\Order_{L}$ of $L$ as a profinite $W(k')$-algebra
to define $\alg{O}_{L, k'}$ to be $\Grn_{k'} \Order_{L}$, which is a sheaf of rings on $\pfpqc{k'}$.
On the other hand, the ring $\Order_{L}$ can be regarded as a profinite $W(k)$-algebra,
so that we can define another sheaf of rings on $\pfpqc{k}$ to be $\Grn_{k} \Order_{L}$,
which we denote by $\alg{O}_{L, k}$.
The inclusion $\Order_{K} \into \Order_{L}$ is a morphism of profinite $W(k)$-algebras,
so it induces an inclusion $\alg{O}_{K} \into \alg{O}_{L, k}$ of perfect affine $W^{(\infty)}$-algebras over $k$.
For a perfect $k$-algebra $R$,
we have $\alg{O}_{L, k}(R) = \alg{O}_{K}(R) \tensor_{\Order_{K}} \Order_{L}$.
Hence we can define the norm map $N_{L / K} \colon \alg{O}_{L, k} \to \alg{O}_{K}$
as a morphism of sheaves on $\pfpqc{k}$.
We define sheaves $\alg{U}_{L, k}$ and $\alg{L}_{k}$ on $\pfpqc{k}$
by setting $\alg{U}_{L, k}(R) = \alg{O}_{L, k}(R)^{\times}$
and $\alg{L}_{k}(R) = \alg{O}_{L, k}(R) \tensor_{\Order_{L}} L$ for each perfect $k$-algebra $R$.
When $L / K$ is totally ramified,
these two constructions give the same sheaves
$\alg{O}_{L, k'} = \alg{O}_{L, k}$ and $\alg{L}_{k'} = \alg{L}_{k}$,
so that we can omit the subscript $k = k'$ without ambiguity.

If $L / K$ is a finite Galois extension,
the Galois group $G = \Gal(L / K)$ acts on $\Order_{L}$ over $W(k)$
and the inertia group $T = T(L / K)$ acts on $\Order_{L}$ over $W(k')$.
By the functoriality of $\Grn_{k}$ and $\Grn_{k'}$,
the sheaves $\alg{O}_{L, k}$ and $\alg{L}_{k}$ become sheaves of $G$-modules on $\pfpqc{k}$
and the sheaves $\alg{O}_{L, k'}$ and $\alg{L}_{k'}$ become sheaves of $T$-modules on $\pfpqc{k'}$.
The norm map $N_{L / K}$ coincides with the action of the element
$N_{G} = \sum_{\sigma \in G} \sigma$ of the group ring $\Z[G]$.

We return to a general finite extension $L / K$.
We describe rational points of the sheaves defined above.
We have $\alg{O}_{L, k}(k) = \alg{O}_{L, k'}(k') = \Order_{L}$
and $\alg{L}_{k}(k) = \alg{L}_{k'}(k') = L$.
Also we have
	$
			\alg{O}_{L, k}(\algcl{k})
		=
			W(\algcl{k}) \ctensor_{W(k)} \Order_{L}
		=
			\hat{\Order}_{K}^{\ur} \tensor_{\Order_{K}} \Order_{L}
	$.
To make this more explicit, let $M$ be the maximal unramified subextension of $L / K$.
For a $k$-embedding $\rho \colon k' \into \algcl{k}$,
we denote by $\hat{\Order}_{K}^{\ur} \tensor_{\Order_{M}}^{\rho} \Order_{L}$
the tensor product of $\hat{\Order}_{K}^{\ur}$ and $\Order_{L}$ over $\Order_{M}$
with $\hat{\Order}_{K}^{\ur}$ regarded as an $\Order_{M}$-algebra
via the $\Order_{K}$-embedding $\Order_{M} \into \hat{\Order}_{K}^{\ur}$
that is the lift of $\rho$.
Then the natural map
	$
			\hat{\Order}_{K}^{\ur} \tensor_{\Order_{K}} \Order_{L}
		\to
			\prod_{\rho \in \Hom_{k}(k', \algcl{k})}
				\hat{\Order}_{K}^{\ur} \tensor_{\Order_{M}}^{\rho} \Order_{L}
	$
sending $a \tensor b$ to $(a \tensor^{\rho} b)_{\rho}$ is a ring isomorphism.
This isomorphism translates the action of an element $\rho' \in \Gal(\algcl{k} / k)$
into the automorphism
	$
			(\alpha_{\rho})_{\rho}
		\mapsto
			((\rho' \tensor \id_{\Order_{L}})(\alpha_{\rho'^{-1} \rho}))_{\rho}
	$,
where
	$
			\rho' \tensor \id_{\Order_{L}}
		\colon
			\hat{\Order}_{K}^{\ur} \tensor_{\Order_{M}}^{\rho'^{-1} \rho} \Order_{L}
		\isomto
			\hat{\Order}_{K}^{\ur} \tensor_{\Order_{M}}^{\rho} \Order_{L}
	$
is the isomorphism that sends
$a \tensor^{\rho'^{-1} \rho} b$ to $\rho'(a) \tensor^{\rho} b$.%
	\footnote{This map is well-defined since if $c \in \Order_{M}$, then
	the two different expressions of the same element
		$
				1 \tensor^{\rho'^{-1} \rho} c
			=
				(\rho'^{-1} \rho)(c) \tensor^{\rho'^{-1} \rho} 1
		$
	are mapped to the same element
		$
				1 \tensor^{\rho} c
			=
				\rho(c) \tensor^{\rho} 1
		$.}
If $L / K$ is Galois,
the same isomorphism translates the action of an element $\sigma \in \Gal(L / K)$
into the automorphism
	$
			(\alpha_{\rho})_{\rho}
		\mapsto
			((\id_{\hat{\Order}_{L}^{\ur}} \tensor \sigma)(\alpha_{\rho \sigma|_{M}}))_{\rho}
	$,
where
	$
			\id_{\hat{\Order}_{L}^{\ur}} \tensor \sigma
		\colon
			\hat{\Order}_{K}^{\ur} \tensor_{\Order_{M}}^{\rho \sigma|_{M}} \Order_{L}
		\isomto
			\hat{\Order}_{K}^{\ur} \tensor_{\Order_{M}}^{\rho} \Order_{L}
	$
is the isomorphism that sends
$a \tensor^{\rho \sigma|_{M}} b$ to $a \tensor^{\rho} \sigma(b)$.%
	\footnote{Similarly, this map is well-defined since if $c \in \Order_{M}$, then
	the two different expressions of the same element
		$
				1 \tensor^{\rho \sigma|_{M}} c
			=
				(\rho \sigma)(c) \tensor^{\rho \sigma|_{M}} 1
		$
	are mapped to the same element
		$
				1 \tensor^{\rho} \sigma(c)
			=
				(\rho \sigma)(c) \tensor^{\rho} 1
		$.}
For each $\rho \in \Hom_{k}(k', \algcl{k})$,
the ring $\hat{\Order}_{K}^{\ur} \tensor_{\Order_{M}}^{\rho} \Order_{L}$
is non-canonically isomorphic to $\hat{\Order}_{L}^{\ur}$.
Similarly we have
	$
			\alg{L}_{k}(\algcl{k})
		=
			\hat{K}^{\ur} \tensor_{K} L
		\isomto
			\prod_{\rho}
				\hat{K}^{\ur} \tensor_{M}^{\rho} L
		\cong
			(\hat{L}^{\ur})^{\Hom_{k}(k', \algcl{k})}
	$,
the last isomorphism being non-canonical.

We discuss the valuation map for $L$.
We apply $\Res_{k' / k}$ to the split exact sequence
$0 \to \alg{U}_{L, k'} \to \alg{L}_{k'}^{\times} \to \Z \to 0$ in $\psh{k'}$.
By Proposition \ref{Prop:GrnRes},
we have $\Res_{k' / k} \alg{O}_{L, k'} \cong \alg{O}_{L, k}$,
and so $\Res_{k' / k} \alg{L}_{k'} \cong \alg{L}_{k}$.
The sheaf $\Res_{k' / k} \Z$ is the \'etale group scheme over $k$
whose group of $\algcl{k}$-points is the $\Gal(\algcl{k} / k)$-module $\Z[\Hom_{k}(k', \algcl{k})]$,
the free abelian group generated by the $\Gal(\algcl{k} / k)$-set $\Hom_{k}(k', \algcl{k})$.
Therefore we have a split exact sequence
$0 \to \alg{U}_{L, k} \to \alg{L}_{k}^{\times} \to \Z[\Hom_{k}(k', \algcl{k})] \to 0$ in $\psh{k}$.
The map $(\alg{L}_{k}^{\times})(\algcl{k}) \to \Z[\Hom_{k}(k', \algcl{k})]$
is translated, via the above description, to the map
$((\hat{L}^{\ur})^{\times})^{\Hom_{k}(k', \algcl{k})} \to \Z[\Hom_{k}(k', \algcl{k})]$
sending $(\alpha_{\rho})_{\rho} \mapsto \sum_{\rho} v_{\hat{L}^{\ur}}(\alpha_{\rho}) \rho$.
We have a commutative diagram with exact rows
	\begin{equation} \label{Eq:ValDiagram}
		\begin{CD}
				0
			@>>>
				\alg{U}_{K}
			@>>>
				\alg{K}^{\times}
			@>>>
				\Z
			@>>>
				0
			\\
			@.
			@VV \text{incl} V
			@VV \text{incl} V
			@VVV
			@.
			\\
				0
			@>>>
				\alg{U}_{L, k}
			@>>>
				\alg{L}_{k}^{\times}
			@>>>
				\Z[\Hom_{k}(k', \algcl{k})]
			@>>>
				0
			\\
			@.
			@VV N_{L / K} V
			@VV N_{L / K} V
			@VVV
			@.
			\\
				0
			@>>>
				\alg{U}_{K}
			@>>>
				\alg{K}^{\times}
			@>>>
				\Z
			@>>>
				0,
		\end{CD}
	\end{equation}
where the first morphism at the right column sends $1$ to $\sum_{\rho} \rho$
and the second sends every $\rho$ to $1$.
We define $\alg{U}_{L,\, k' / k}$ to be the kernel of the composite map
$\alg{L}_{k}^{\times} \onto \Z[\Hom_{k}(k', \algcl{k})] \onto \Z$.
We have an exact sequence
	\begin{equation} \label{Eq:ValExtSh}
			0
		\to
			\alg{U}_{L,\, k' / k}
		\to
			\alg{L}_{k}^{\times}
		\to
			\Z
		\to
			0.
	\end{equation}
This sequence is the one we will use instead of the sequence \eqref{Eq:ValExtGp}.
If $L / K$ is Galois, then the inclusions
$\alg{U}_{L, k} \into \alg{U}_{L,\, k' / k} \into \alg{L}_{k}^{\times}$ are
morphisms of sheaves of $G$-modules.
Thus we have an action of $G$ on 
$\Coker(\alg{U}_{L, k} \into \alg{L}_{k}^{\times}) = \Z[\Gal(k' / k)]$.
The action of an element $\sigma \in G$ on $\Z[\Gal(k' / k)]$ is given
by multiplication by the image of $\sigma^{-1}$ from the right.


\subsection{Proof of the main theorem} \label{Sec:LCFT}
We prove Theorem \ref{Th:Main}.
We need the following vanishing result.

\begin{prp} \label{Prop:Vanish}
	Let $L / K$ be a finite Galois extension.
	Then the Tate cohomology sheaf $\hat{H}^{i}(\Gal(L / K), \alg{L}_{k}^{\times})$
	vanishes for all $i \in \Z$.
	More generally, for any subextension $E$ of $L / K$,
	the sheaf $\hat{H}^{i}(\Gal(L / E), \alg{L}_{k}^{\times})$
	vanishes for all $i \in \Z$.
\end{prp}

\begin{proof}
	First we show that
	$\hat{H}^{i}(\Gal(L / K), \alg{L}_{k}^{\times}) = 0$.
	By the exact sequence
	$0 \to \alg{U}_{L, k} \to \alg{L}_{k}^{\times} \to \Z[\Gal(k' / k)] \to 0$
	and Proposition \ref{Prop:Tate}, we know that
	the sheaf $\hat{H}^{i}(\Gal(L / K), \alg{L}_{k}^{\times})$ is an affine proalgebraic group
	with group of $\algcl{k}$-points given by
		$
				\hat{H}^{i}(\Gal(L / K), \alg{L}_{k}^{\times}(\algcl{k}))
			=
				\hat{H}^{i}(\Gal(L / K), (\hat{K}^{\ur} \tensor_{K} L)^{\times})
		$.
	The description of $(\hat{K}^{\ur} \tensor_{K} L)^{\times}$ in Section \ref{Sec:SheavesL}
	shows that this $\Gal(L / K)$-module is induced
	from the $T(L / K)$-module $(\hat{L}^{\ur})^{\times}$.
	By Shapiro's lemma (\cite[VII, \S 5, Exercise]{Ser79}),
	we have
		$
				\hat{H}^{i}(\Gal(L / K), (\hat{K}^{\ur} \tensor_{K} L)^{\times})
			=
				\hat{H}^{i}(T(L / K), (\hat{L}^{\ur})^{\times})
		$.
	This group is zero as shown in the proof of Theorem \ref{Th:SerreLCFT}.
	
	Next we show that
	$\hat{H}^{i}(\Gal(L / E), \alg{L}_{k}^{\times}) = 0$.
	The isomorphism $\alg{L}_{k}^{\times} \cong \Res_{k'' / k} \alg{L}_{k''}^{\times}$
	stated in Section \ref{Sec:SheavesL}
	is $\Gal(L / E)$-equivariant.
	Since $\Res_{k'' / k}$ is an exact functor by part \ref{Enum:WeilAdj} of Proposition \ref{Prop:Weil},
	we have
		$
				\hat{H}^{i}(\Gal(L / E), \alg{L}_{k}^{\times})
			\cong
				\Res_{k'' / k} \hat{H}^{i}(\Gal(L / E), \alg{L}_{k''}^{\times})
		$,
	which is zero by the first case.
\end{proof}

\begin{proof}[Proof of Theorem \ref{Th:Main}]
	First we construct a homomorphism
	$\pi_{1}^{k}(\alg{K}^{\times}) \to \Gal(L / K)^{\ab}$
	for each finite Galois extension $L / K$.
	Let $k'$ be the residue field of $L$.
	We set $G = \Gal(L / K)$, $\mathfrak{g} = \Gal(k' / k)$,
	$T = T(L / K)$ and $T_{a} = T(L \cap K^{\ab} / K)$.
	We regard $G$, $\mathfrak{g}$ and $T_{a}$ as constant groups over $k$
	(though the group ring $\Z[\mathfrak{g}]$ is
	regarded as an \'etale group over $k$ as before).
	We apply Proposition \ref{Prop:Vanish} to
	the short exact sequence \eqref{Eq:ValExtSh} of sheaves of $G$-modules.
	The long exact sequence then gives an isomorphism
	$\hat{H}^{i - 1}(G, \Z) \isomto \hat{H}^{i}(G, \alg{U}_{L,\, k' / k})$
	for any $i \in \Z$.
	
	We examine this isomorphism for $i = 0$.
	Since $\hat{H}^{-1}(G, \Z) = 0$,
	we have $\hat{H}^{0}(G, \alg{U}_{L,\, k' / k}) = 0$.
	This means, by the definition of Tate cohomology, that
	the norm map (endomorphism)
	$N = \sum_{\sigma \in G} \sigma \colon \alg{U}_{L,\, k' / k} \to \alg{U}_{L,\, k' / k}$
	is a surjection onto the $G$-invariant part of the sheaf of $G$-modules $\alg{U}_{L,\, k' / k}$.
	Since the $G$-invariant part of the morphism $\alg{L}_{k}^{\times} \onto \Z[\mathfrak{g}]$
	is the valuation map $\alg{K}^{\times} \onto \Z$,
	the $G$-invariant part of $\alg{U}_{L,\, k' / k}$ is $\alg{U}_{K}$.
	Hence the norm map $N$ gives a surjection
	$\alg{U}_{L,\, k' / k} \onto \alg{U}_{K}$.
	
	Next we examine the same isomorphism for $i = -1$.
	Since $\hat{H}^{-2}(G, \Z)$ is the constant group $G^{\ab}$,
	we have $\hat{H}^{-1}(G, \alg{U}_{L,\, k' / k}) \cong G^{\ab}$.
	By definition, the sheaf
	$\hat{H}^{-1}(G, \alg{U}_{L,\, k' / k})$
	is the kernel of the norm map
	$N \colon \alg{U}_{L,\, k' / k} \onto \alg{U}_{K}$
	divided by the product $I_{G} \alg{U}_{L,\, k' / k}$
	of the sheaf of $G$-modules $\alg{U}_{L,\, k' / k}$ and
	the augmentation ideal $I_{G}$ of the group ring $\Z[G]$.
	Therefore we get a short exact sequence
		$
				0
			\to
				G^{\ab}
			\to
				\alg{U}_{L,\, k' / k} / I_{G} \alg{U}_{L,\, k' / k}
			\overset{N}{\to}
				\alg{U}_{K}
			\to
				0
		$.
	
	Consider the following commutative diagram with exact rows:
		\[
			\begin{CD}
					0
				@>>>
					\alg{U}_{L,\, k' / k} / I_{G} \alg{U}_{L,\, k' / k}
				@>>>
					\alg{L}_{k}^{\times} / I_{G} \alg{U}_{L,\, k' / k}
				@>>>
					\Z
				@>>>
					0
				\\
				@.
				@VV N V
				@VV N V
				@|
				@.
				\\
					0 @>>> \alg{U}_{K} @>>> \alg{K}^{\times} @>>> \Z @>>> 0.
			\end{CD}
		\]
	The above short exact sequence fits in the first column of this diagram.
	Hence we get a short exact sequence
		\[
				0
			\to
				G^{\ab}
			\to
				\alg{L}_{k}^{\times} / I_{G} \alg{U}_{L,\, k' / k}
			\overset{N}{\to}
				\alg{K}^{\times}
			\to
				0.
		\]
	The resulting long exact sequence of homotopy groups gives a homomorphism
		\[
			\pi_{1}^{k}(\alg{K}^{\times}) \to G^{\ab}.
		\]
	Note that we can give the following more explicit description of the morphism
	$G^{\ab} \to \alg{L}_{k}^{\times} / I_{G} \alg{U}_{L,\, k' / k}$
	via the presentation 
		$
				\alg{L}_{k}^{\times}(\algcl{k})
			=
				\prod_{\rho \in \mathfrak{g}}
					(\hat{K}^{\ur} \tensor_{M}^{\rho} L)^{\times}
		$
	given in Section \ref{Sec:SheavesL}.
	Let
		$
				\alpha
			=
				(\alpha_{\rho})_{\rho}
			\in
				\prod_{\rho \in \mathfrak{g}}
					(\hat{K}^{\ur} \tensor_{M}^{\rho} L)^{\times}
		$
	be the element given by
	$\alpha_{\rho} = 1$ for $\rho \ne \id$
	and $\alpha_{\id} = 1 \tensor \pi_{L}$ for a prime element $\pi_{L}$ of $\Order_{L}$.%
		\footnote{This $\alpha$ is different from
		$\beta := 1 \tensor \pi_{L} \in (\hat{K}^{\ur} \tensor_{K} L)^{\times}$
		unless $L / K$ is totally ramified,
		since $\beta$ corresponds to
			$
					(1 \tensor^{\rho} \pi_{L})_{\rho}
				\in
					\prod_{\rho \in \mathfrak{g}}
						(\hat{K}^{\ur} \tensor_{M}^{\rho} L)^{\times}
			$,
		whose $\rho$-component for any $\rho \ne \id$ is $1 \tensor^{\rho} \pi_{L} \ne 1$.
		Also if $L / K$ is unramified and $\pi_{L}$ is taken from $K$,
		then $\sigma(\beta) = \beta$ and $\sigma(\beta) / \beta = 1 \ne \sigma(\alpha) / \alpha$
		unless $L = K$.}
	Then, by writing down all maps involved, we see that
	the image of $\sigma \in G^{\ab}$ is given by $\sigma(\alpha) / \alpha$
	(see Section \ref{Sec:SheavesL} for the description of $\sigma(\alpha)$).
	This element $\sigma(\alpha) / \alpha$ is further mapped to $\sigma^{-1} - 1$
	via the valuation map $\alg{L}_{k}^{\times} \onto \Z[\mathfrak{g}]$.
	
	Next we show that the homomorphisms
	$\pi_{1}^{k}(\alg{K}^{\times}) \to G^{\ab}$ just constructed
	form an inverse system for finite Galois extensions $L / K$.
	Let $L_{1}$ be a finite Galois extension of $K$ containing $L$.
	Let $k_{1}'$ be the residue field of $L_{1}$ and set $G_{1} = \Gal(L_{1} / K)$.
	To show that the homomorphisms $\pi_{1}^{k}(\alg{K}^{\times}) \to G_{1}^{\ab}$ and
	$\pi_{1}^{k}(\alg{K}^{\times}) \to G^{\ab}$ are compatible,
	consider the following commutative diagram with exact rows:
		\[
			\begin{CD}
					0
				@>>>
					\alg{U}_{L_{1},\, k'_{1} / k}
				@>>>
					\alg{L}_{1, k}^{\times}
				@>>>
					\Z
				@>>>
					0
				\\
				@.
				@VV N_{L_{1} / L} V
				@VV N_{L_{1} / L} V
				@|
				@.
				\\
					0
				@>>>
					\alg{U}_{L,\, k' / k}
				@>>>
					\alg{L}_{k}^{\times}
				@>>>
					\Z
				@>>>
					0.
			\end{CD}
		\]
	The top row is a sequence of sheaves of $G_{1}$-modules
	and the bottom row is a sequence of sheaves of $G$-modules.
	The actions are compatible with the natural surjection $G_{1} \onto G$.
	This diagram induces a commutative diagram on homology
		\[
			\begin{CD}
					H_{1}(G_{1}, \Z)
				@>>>
					H_{0}(G_{1}, \alg{U}_{L_{1},\, k'_{1} / k})
				\\
				@| @VV N_{L_{1} / L} V
				\\
					H_{1}(G_{1}, \Z)
				@>>>
					H_{0}(G_{1}, \alg{U}_{L,\, k' / k})
				\\
				@VVV @VVV
				\\
					H_{1}(G, \Z)
				@>>>
					H_{0}(G, \alg{U}_{L,\, k' / k}).
			\end{CD}
		\]
	Note that $N_{L_{1} / L}$ maps the kernel of $N_{G_{1}} = N_{L_{1} / K}$ on $\alg{U}_{L_{1},\, k' / k}$
	to the kernel of $N_{G} = N_{L / K}$ on $\alg{U}_{L,\, k' / k}$
	since $N_{L_{1} / K} = N_{L / K} \compose N_{L_{1} / L}$.
	Therefore we have a commutative diagram
		\[
			\begin{CD}
					\hat{H}^{-2}(G_{1}, \Z)
				@> \sim >>
					\hat{H}^{-1}(G_{1}, \alg{U}_{L_{1},\, k'_{1} / k})
				\\
				@VVV
				@VV N_{L_{1} / L} V
				\\
					\hat{H}^{-2}(G, \Z)
				@> \sim >>
					\hat{H}^{-1}(G, \alg{U}_{L,\, k' / k}).
			\end{CD}
		\]
	The left vertical map is identified with the natural surjection
	$G_{1}^{\ab} \onto G^{\ab}$.
	Therefore we have a commutative diagram with exact rows
		\[
			\begin{CD}
					0
				@>>>
					G_{1}^{\ab}
				@>>>
					\alg{L}_{1, k}^{\times} / I_{G_{1}} \alg{U}_{L_{1},\, k'_{1} / k}
				@> N_{L_{1} / K} >>
					\alg{K}^{\times}
				@>>>
					0
				\\
				@.
				@VV \mathrm{can.} V
				@VV N_{L_{1} / L} V
				@|
				@.
				\\
					0
				@>>>
					G^{\ab}
				@>>>
					\alg{L}_{k}^{\times} / I_{G} \alg{U}_{L,\, k' / k}
				@> N_{L / K} >>
					\alg{K}^{\times}
				@>>>
					0.
			\end{CD}
		\]
	The resulting long exact sequences of homotopy groups show the compatibility.
	
	Hence we have obtained a homomorphism
	$\pi_{1}^{k}(\alg{K}^{\times}) \to \Gal(K^{\ab} / K)$.
	We show that this satisfies the commutative diagram \eqref{Eq:LCFT:Compati} in the theorem.
	Let $I_{\mathfrak{g}}$ be the augmentation ideal of the group ring $\Z[\mathfrak{g}]$.
	Since
		$
				\mathfrak{g}^{\ab}
			\cong 
				I_{\mathfrak{g}} / I_{\mathfrak{g}}^{2}
		$
	by $\rho \leftrightarrow \rho - 1$,
	we have an exact sequence
		$
			0 \to \mathfrak{g}^{\ab} \to \Z[\mathfrak{g}] / I_{\mathfrak{g}}^{2} \to \Z \to 0
		$.
	Also the surjection $\alg{L}_{k}^{\times} \onto \Z[\mathfrak{g}]$
	gives a surjection $\alg{U}_{L,\, k' / k} \onto I_{\mathfrak{g}}$.
	Multiplying $I_{G}$, we have a surjection
	$I_{G} \alg{U}_{L,\, k' / k} \onto I_{G} I_{\mathfrak{g}} = I_{\mathfrak{g}}^{2}$.
	Therefore the kernel of the surjection
		$
				\alg{L}_{k}^{\times} / I_{G} \alg{U}_{L,\, k' / k}
			\onto
				\Z[\mathfrak{g}] / I_{\mathfrak{g}}^{2}
		$
	is $\alg{U}_{L, k} / (\alg{U}_{L, k} \cap I_{G} \alg{U}_{L,\, k' / k})$.
	Hence we have the following commutative diagram with exact rows and columns:
		\begin{equation}
			\begin{CD}
				@. 0 @. 0 @. 0 @.
				\\
				@. @VVV @VVV @VVV @.
				\\
					0
				@>>>
					T_{a}
				@>>>
					\alg{U}_{L, k} / (\alg{U}_{L, k} \cap I_{G} \alg{U}_{L,\, k' / k})
				@>>>
					\alg{U}_{K}
				@>>>
					0
				\\
				@. @VVV @VVV @VVV @.
				\\
					0
				@>>>
					G^{\ab}
				@>>>
					\alg{L}_{k}^{\times} / I_{G} \alg{U}_{L,\, k' / k}
				@>>>
					\alg{K}^{\times}
				@>>>
					0
				\\
				@. @VVV @VVV @VVV @.
				\\
					0
				@>>>
					\mathfrak{g}^{\ab}
				@>>>
					\Z[\mathfrak{g}] / I_{\mathfrak{g}}^{2}
				@>>>
					\Z
				@>>>
					0
				\\
				@. @VVV @VVV @VVV @.
				\\
				@. 0 @. 0 @. 0. @.
			\end{CD}
			\label{Eq:NormDiagram}
		\end{equation}
	Here, since the composite map
		$
				G^{\ab} 
			\to
				\alg{L}_{k}^{\times} / I_{G} \alg{U}_{L,\, k' / k}
			\to
				\Z[\mathfrak{g}] / I_{\mathfrak{g}}^{2}
		$
	sends $\sigma$ to $\sigma^{-1} - 1 = 1 - \sigma$
	as we saw after the construction of $\pi_{1}^{k}(\alg{K}^{\times}) \to G^{\ab}$,
	the homomorphism $\mathfrak{g}^{\ab} \to \Z[\mathfrak{g}]$
	is switched from the natural one $\rho \to \rho - 1$ to $-1$ times it.
	The resulting long exact sequences of homotopy groups give
	the commutative diagram \eqref{Eq:LCFT:Compati}.
	
	We show that the left vertical map of the diagram \eqref{Eq:LCFT:Compati}
	coincides with the isomorphism of the local class field theory of Hazewinkel.
	If $L / K$ is totally ramified,
	then the top horizontal sequence of \eqref{Eq:NormDiagram} becomes
	$0 \to G^{\ab} \to \alg{U}_{L} / I_{G} \alg{U}_{L} \to \alg{U}_{K} \to 0$.
	The morphism $G^{\ab} \to \alg{U}_{L} / I_{G} \alg{U}_{L}$
	sends $\sigma \mapsto \sigma(\pi_{L}) / \pi_{L}$ for a prime element $\pi_{L}$ of $\Order_{L}$
	as we saw after the construction of the homomorphism $\pi_{1}^{k}(\alg{K}^{\times}) \to G^{\ab}$.
	Therefore our sequence
	$0 \to G^{\ab} \to \alg{U}_{L} / I_{G} \alg{U}_{L} \to \alg{U}_{K} \to 0$
	for totally ramified $L / K$
	coincides with the sequence \eqref{Eq:SerEx},
	so we see the coincidence.
	
	The left and right vertical arrows of the diagram \eqref{Eq:LCFT:Compati} are isomorphisms.
	Hence so is the middle.
\end{proof}


\subsection{Auxiliary results} \label{Sect:Aux}

Propositions \ref{Prop:BC}, \ref{Prop:RamFil} and \ref{Prop:NormCok} below
are originally Lemmas 4.3, 4.4 and 4.5, respectively, of \cite[\S 4]{SY10}.

\begin{prp} \label{Prop:BC}
	Let $E / K$ be a finite extension with residue extension $k'' / k$.
	Then the isomorphisms of Theorem \ref{Th:Main} for $K$ and $E$
	satisfy the following commutative diagram:
		\[
			\begin{CD}
					\pi_{1}^{k}(\alg{E}_{k}^{\times})
				@> \pi_{1}^{k}(N_{E / K}) >>
					\pi_{1}^{k}(\alg{K}^{\times})
				@> \partial >>
					\pi_{0}^{k}(\Ker(N_{E / K}))
				@>>>
					0
				\\
				@VV \wr V
				@VV \wr V
				@VV \wr V
				@.
				\\
					G_{E}^{\ab}
				@>> \mathrm{Res} >
					G_{K}^{\ab}
				@>>>
					\Gal(E \cap K^{\ab} / K)
				@>>>
					0.
			\end{CD}
		\]
	Here we identified $\pi_{1}^{k''}(\alg{E}_{k''}^{\times})$ with
	$\pi_{1}^{k}(\alg{E}_{k}^{\times})$ by part \ref{Enum:HomotInv} of Proposition \ref{Prop:Weil}.
	The map $\partial$ is the boundary map of the homotopy long exact sequence
	coming from the short exact sequence
				$
						0
					\to
						\Ker(N_{E / K})
					\to
						\alg{E}_{k}^{\times}
					\overset{N_{E / K}}{\longrightarrow}
						\alg{K}^{\times}
					\to
						0
				$.
\end{prp}

\begin{proof}
	First we show that $\partial$ is surjective.
	Consider the long exact sequence
		\[
				\cdots
			\to
				\pi_{1}^{k}(\alg{E}_{k}^{\times})
			\to
				\pi_{1}^{k}(\alg{K}^{\times})
			\overset{\partial}{\to}
				\pi_{0}^{k}(\Ker(N_{E / K}))
			\to
				\pi_{0}^{k}(\alg{E}_{k}^{\times})
			\to
				\pi_{0}^{k}(\alg{K}^{\times})
			\to
				0.
		\]
	It is enough to show that the last map
		$
				\pi_{0}^{k}(\alg{E}_{k}^{\times})
			\to
				\pi_{0}^{k}(\alg{K}^{\times})
		$
	is an isomorphism.
	This follows from the bottom half of the diagram \eqref{Eq:ValDiagram}
	by noticing that $\alg{U}_{K}$ and $\alg{U}_{E, k}$ are connected.
	
	Next we show that the left square of the diagram is commutative.
	We may assume $E$ is either separable or purely inseparable.
	
	First we treat the case $E$ is separable.
	Let $L$ be a finite Galois extension of $K$ containing $E$
	and let $k'$ be the residue field of $L$.
	We set $G = \Gal(L / K)$ and $H = \Gal(L / E)$.
	It is enough to show that the diagram
		\[
			\begin{CD}
					\pi_{1}^{k}(\alg{E}_{k}^{\times})
				@> \pi_{1}^{k}(N_{E / K}) >>
					\pi_{1}^{k}(\alg{K}^{\times})
				\\
				@VVV
				@VVV
				\\
					H^{\ab}
				@> \mathrm{can.} >>
					G^{\ab}
			\end{CD}
		\]
	is commutative.
	For this, it suffices to construct a diagram
		\[
			\begin{CD}
					0
				@>>>
					H^{\ab}
				@>>>
					\alg{L}_{k}^{\times} / I_{H} \alg{U}_{L, k' / k}
				@> N_{L / E} >>
					\alg{E}_{k}^{\times}
				@>>>
					0
				\\
				@.
				@VV \mathrm{can.} V 
				@VV \mathrm{can.} V
				@VV N_{E / K} V
				@.
				\\
					0
				@>>>
					G^{\ab}
				@>>>
					\alg{L}_{k}^{\times} / I_{G} \alg{U}_{L, k' / k}
				@> N_{L / K} >>
					\alg{K}^{\times}
				@>>>
					0
			\end{CD}
		\]
	and prove the commutativity of the squares and the exactness of the rows.
	We construct the top row.
	We regard the exact sequence
	$0 \to \alg{U}_{L, k' / k} \to \alg{L}_{k}^{\times} \to \Z \to 0$
	of \eqref{Eq:ValExtSh} as an exact sequence of sheaves of $H$-modules.
	By Proposition \ref{Prop:Vanish},
	we have $\hat{H}^{i}(H, \alg{L}_{k}^{\times}) = 0$.
	Therefore we have $\hat{H}^{i - 1}(H, \Z) \isomto \hat{H}^{i}(H, \alg{U}_{L,\, k' / k})$.
	The $H$-invariant part of the morphism
	$\alg{L}_{k}^{\times} \onto \Z$ is $\alg{E}_{k}^{\times} \onto \Z$,
	which implies that the $H$-invariant part of
	$\alg{U}_{L,\, k' / k}$ is $\alg{U}_{E,\, k'' / k}$.
	The rest of the construction of the top row is the same as that of the bottom row,
	which we did in the previous section.
	The commutativity of the left square follows from
	the naturality of corestriction maps
		\[
			\begin{CD}
					\hat{H}^{-2}(H, \Z)
				@> \sim >>
					\hat{H}^{-1}(H, \alg{U}_{L,\, k' / k})
				\\
				@V \Cores VV
				@VV \Cores V
				\\
					\hat{H}^{-2}(G, \Z)
				@> \sim >>
					\hat{H}^{-1}(G, \alg{U}_{L,\, k' / k}).
			\end{CD}
		\]
	
	Next we treat the case $E / K$ is purely inseparable.
	Let $L / K$ be a finite Galois extension with residue extension $k' / k$.
	We set $F = L E$.
	Then $\Gal(L / K) = \Gal(F / E)$.
	We have a commutative diagram
		\[
			\begin{CD}
					0
				@>>>
					\alg{U}_{F, k' / k}
				@>>>
					\alg{F}_{k}^{\times}
				@>>>
					\Z
				@>>>
					0
				\\
				@.
				@VV N_{F / L} V 
				@VV N_{F / L} V
				@|
				@.
				\\
					0
				@>>>
					\alg{U}_{L, k' / k}
				@>>>
					\alg{L}_{k}^{\times}
				@>>>
					\Z
				@>>>
					0.
			\end{CD}
		\]
	The rest of the proof is easy and similar to the separable case.
\end{proof}

If $k$ is quasi-finite with given generator $F$ of its absolute Galois group,
the above proposition implies that
the homomorphism $K^{\times} \to \pi_{1}^{k}(\alg{K}^{\times})$ of \eqref{Eq:HomOnRatF}
followed by the isomorphism $\pi_{1}^{k}(\alg{K}^{\times}) \isomto \Gal(K^{\ab} / K)$
gives a homomorphism $K^{\times} / N_{E / K} E^{\times} \to \Gal(E \cap K^{\ab} / K)$.
This and the diagram \eqref{Eq:LCFT:Compati} together imply that
our $K^{\times} \to \Gal(K^{\ab} / K)$ has to be the same as
the canonical homomorphism of the usual local class field theory times $-1$,
which sends a prime element to an automorphism that acts on $k^{\ab} = \algcl{k}$ by $F^{-1}$.

\begin{prp} \label{Prop:RamFil}
	Let $L / K$ be a finite totally ramified abelian extension
	with Galois group $G$.
	Let $\psi = \psi_{L / K}$ be the Herbrand function,
	$m \ge 1$ an integer, $N \colon \alg{U}_{L} \to \alg{U}_{K}$ the norm map
	and
		$
				\bar{N}
			\colon
				\alg{U}_{L} / \alg{U}_{L}^{\psi(m - 1) + 1}
			\to
				\alg{U}_{K} / \alg{U}_{K}^{m}
		$
	its quotient.
	Then we have $\pi_{0}^{k}(\Ker(\bar{N})) \cong G / G^{m}$,
	where $G^{m}$ is the $m$-th ramification group in the upper numbering.
\end{prp}

\begin{proof}
	By \cite[\S 3.4, Prop.\ 6 (a)]{Ser61},
	we have $N(\alg{U}_{L}^{\psi(m - 1) + 1}) = \alg{U}_{K}^{m}$.
	This and a diagram chase show that the commutative diagram
		\[
			\begin{CD}
					0
				@>>>
					\Ker(N) \cap \alg{U}_{L}^{\psi(m - 1) + 1}
				@>>>
					\Ker(N) \cap \alg{U}_{L}^{\psi(m - 1)}
				@>>>
					\Ker(\bar{\bar{N}})
				@>>>
					0
				\\
				@. @| @VVV @VVV @.
				\\
					0
				@>>>
					\Ker(N) \cap \alg{U}_{L}^{\psi(m - 1) + 1}
				@>>>
					\Ker(N)
				@>>>
					\Ker(\bar{N})
				@>>>
					0
			\end{CD}
		\]
	has exact rows,
	where
		$
				\bar{\bar{N}}
			\colon
				\alg{U}_{L}^{\psi(m - 1)} / \alg{U}_{L}^{\psi(m - 1) + 1}
			\to
				\alg{U}_{K}^{m - 1} / \alg{U}_{K}^{m}
		$.
	Apply $\pi_{0}^{k}$ to this diagram.
	We use \cite[\S 3.5, Prop.\ 8, (ii)]{Ser61},
	Proposition \ref{Prop:BC} (or \cite[\S 2.3, Cor.\ to Prop.\ 3]{Ser61})
	and \cite[\S 3.4, Prop.\ 6, (b)]{Ser61}
	for the top middle term, bottom middle term and the top right term respectively.
	Then we get a commutative diagram with exact rows
		\[
			\begin{CD}
					\pi_{0}^{k}(\Ker(N) \cap \alg{U}_{L}^{\psi(m - 1) + 1})
				@>>>
					G^{m - 1}
				@>>>
					G^{m - 1} / G^{m}
				@>>>
					0
				\\
				@| @VVV @VVV @.
				\\
					\pi_{0}^{k}(\Ker(N) \cap \alg{U}_{L}^{\psi(m - 1) + 1})
				@>>>
					G
				@>>>
					\pi_{0}^{k}(\Ker(\bar{N}))
				@>>>
					0.
			\end{CD}
		\]
	Thus we have $\pi_{0}^{k}(\Ker(\bar{N})) \cong G / G^{m}$.
\end{proof}

\begin{prp} \label{Prop:NormCok}
	Let $L / K$, $G$, $\psi$ and $N$ have the same meaning as in the previous proposition.
	The homomorphism $K^{\times} \to \Hom(\Gal(k^{\ab} / k), \pi_{1}^{k}(\alg{K}^{\times}))$
	of \eqref{Eq:HomOnRat} with the isomorphism
	$\pi_{1}^{k}(\alg{K}^{\times}) \isomto \Gal(K^{\ab} / K)$
	induces isomorphisms
		$
				K^{\times} / N L^{\times}
			\cong
				\Hom(\Gal(k^{\ab} / k), G)
		$
	and
		$
				U_{K}^{m - 1} / U_{K}^{m} N U_{L}^{\psi(m - 1)}
			\cong
				\Hom(\Gal(k^{\ab} / k), G^{m - 1} / G^{m})
		$
	for any integer $m \ge 1$.
	If $K' / K$ is a finite unramified extension with residue extension $k' / k$
	and $L' = K' L$,
	then these isomorphisms satisfy commutative diagrams
		\begin{gather*}
				\begin{CD}
						K^{\times} / N L^{\times}
					@=
						\Hom(\Gal(k^{\ab} / k), G)
					\\
					@VVV
					@VVV
					\\
						K'^{\times} / N L'^{\times}
					@=
						\Hom(\Gal(k'^{\ab} / k'), G),
				\end{CD}
			\\
				\begin{CD}
						U_{K}^{m - 1} / U_{K}^{m} N U_{L}^{\psi(m - 1)}
					@=
						\Hom(\Gal(k^{\ab} / k), G^{m - 1} / G^{m})
					\\
					@VVV
					@VVV
					\\
						U_{K'}^{m - 1} / U_{K'}^{m} N U_{L'}^{\psi(m - 1)}
					@=
						\Hom(\Gal(k'^{\ab} / k'), G^{m - 1} / G^{m}).
				\end{CD}
		\end{gather*}
	Here the vertical maps are induced by the inclusion $K^{\times} \into K'^{\times}$
	and the natural map $\Gal(k'^{\ab} / k') \to \Gal(k^{\ab} / k)$.
\end{prp}

\begin{proof}
	First we have $K^{\times} / N L^{\times} = U_{K} / N U_{L}$
	since $L / K$ is totally ramified.
	We have short exact sequences
		\begin{gather*}
					0
				\to
					G
				\to
					\alg{U}_{L} / I_{G} \alg{U}_{L}
				\to
					\alg{U}_{K}
				\to
					0
			\quad \text{and}
			\\
					0
				\to
					G^{m - 1} / G^{m}
				\to
					\alg{U}_{L}^{\psi(m - 1)} / \alg{U}_{L}^{\psi(m - 1) + 1}
				\to
					\alg{U}_{K}^{m - 1} / \alg{U}_{K}^{m}
				\to
					0
		\end{gather*}
	of proalgebraic groups over $k$
	by \eqref{Eq:SerEx} and \cite[\S 3.4, Prop.\ 6 (b)]{Ser61} respectively.
	The coboundary maps of Galois cohomology of $k$%
		\footnote{Again do not confuse this with Tate cohomology sheaves.}
	for these sequences
	give the maps
		$
				U_{K}^{\times} / N U_{L}^{\times}
			\to
				\Hom(G_{k}, G)
		$
	and
		$
				U_{K}^{m - 1} / U_{K}^{m} N U_{L}^{\psi(m - 1)}
			\to
				\Hom(G_{k}, G^{m - 1} / G^{m})
		$
	in the statement by construction.
	We have 
		$
				\alg{U}_{L} / I_{G} \alg{U}_{L}
			=
				\Gm^{(\infty)} \times \alg{U}_{L}^{1} / I_{G} \alg{U}_{L}^{1}
		$.
	The group $\alg{U}_{L}^{1} / I_{G} \alg{U}_{L}^{1}$ is a connected affine unipotent proalgebraic group,
	so it has a filtration with subquotients all isomorphic to $\Ga^{(\infty)}$.
	We have $H^{1}(k, \Gm^{(\infty)}) = H^{1}(k, \Ga^{(\infty)}) = 0$,
	so $H^{1}(k, \alg{U}_{L} / I_{G} \alg{U}_{L}) = 0$.
	Also $\alg{U}_{L}^{\psi(m - 1)} / \alg{U}_{L}^{\psi(m - 1) + 1}$
	is isomorphic to $\Gm^{(\infty)}$ if $m = 1$ and to $\Ga^{(\infty)}$ if $m > 1$.
	Hence we have $H^{1}(k, \alg{U}_{L}^{\psi(m - 1)} / \alg{U}_{L}^{\psi(m - 1) + 1}) = 0$.
	Therefore we get the required isomorphisms.
	The commutativity of the two diagrams are the naturality of corestriction maps.
\end{proof}

This proposition shares large part with Fesenko's result in his paper \cite{Fes93}.


\end{document}